\documentclass{amsart}

\usepackage{amsthm,amscd,amssymb,amsxtra,amsbsy}
\usepackage{url}
\usepackage{verbatim}
\usepackage{graphicx}
\usepackage{hyperref}

\hypersetup{colorlinks=true}

\newtheorem{Thm}{Theorem}
\newtheorem{Prop}[Thm]{Proposition}
\newtheorem{Lem}[Thm]{Lemma}
\newtheorem{Cor}[Thm]{Corollary}
\newtheorem{Conj}[Thm]{Conjecture}

\DeclareMathOperator{\Hom}{Hom}
\DeclareMathOperator{\End}{End}
\DeclareMathOperator{\Aut}{Aut}
\DeclareMathOperator{\Ext}{Ext}

\newcommand{\dT}{\frac{\mathrm d}{\mathrm dT}}

\newcommand{\bQ}{\mathbb Q}
\newcommand{\modcat}{\bmod}
\newcommand{\HA}{\mathcal H}
\newcommand{\Cat}{\mathcal C}

\newcommand{\blambda}{\boldsymbol\lambda}
\newcommand{\bmu}{\boldsymbol\mu}
\newcommand{\bxi}{\boldsymbol\xi}

\title{Ringel-Hall Algebras of Cyclic Quivers}
\author{Andrew Hubery}
\address{University of Leeds\\Leeds\\U.K.}
\email{a.w.hubery@leeds.ac.uk}
\date{}

\begin{document}

\maketitle

\section{Introduction}

The Hall algebra, or algebra of partitions, was originally constructed in the context of abelian $p$-groups, and has a history going back to a talk by Steinitz \cite{Steinitz}. This work was largely forgotten, leaving Hall to rediscover the algebra fifty years later \cite{Hall}. (See also the articles \cite{Johnsen, Macdonald1}.) The Hall algebra is naturally isomorphic to the ring of symmetric functions, and in fact this is an isomorphism of self-dual graded Hopf algebras.

The basic idea is to count short exact sequences with fixed isomorphism classes of $p$-groups and then to use these numbers as the structure constants for an algebra. This idea was picked up again by Ringel \cite{Ringel1} for more general module categories, and in particular the category of finite dimensional representations of a quiver (over a finite field).

Ringel's work built on some remarkable results relating quiver representations to symmetrisable Kac-Moody Lie algebras, beginning with \cite{Gabriel, BGP, Donovan-Freislich, Nazarova, Dlab-Ringel} and culminating in Kac's Theorem \cite{Kac1}, which states that over an algebraically closed field, the dimension vector (or image in the Grothendieck group) gives a surjection from the set of isomorphism classes of indecomposable representations to the set of positive roots of the associated root system. This root system can also be realised as that coming from a symmetrisable Kac-Moody Lie algebra \cite{Kac2}.

This connection was extended by Ringel in a series of papers \cite{Ringel1, Ringel2, Ringel3, Ringel4, Ringel7, Ringel8} where he constructed the Ringel-Hall algebra and studied its properties, in particular proving the existence of Hall polynomials for representation-directed algebras. Moreover, if one specialises these polynomials at 1, then the indecomposable modules yield a Lie subalgebra with universal enveloping algebra the whole Ringel-Hall algebra.

This work was later generalised in two different ways. In \cite{Ringel5, Riedtmann, Schofield} the Lie algebra/universal enveloping algebra approach was taken further, with Riedtmann and Schofield replacing the evaluation of polynomials at 1 with the Euler characteristic of certain varieties. In particular, Schofield proves that one can recover the universal enveloping algebra of an arbitrary symmetric Kac-Moody Lie algebra by studying the variety of quiver representations over the field of complex numbers.

On the other hand, Green proved in \cite{Green1} that the Ringel-Hall algebra can be endowed with a comultiplication such that it becomes a twisted bialgebra. He then related the composition subalgebra to the positive part of the quantum group for the corresponding symmetrisable Kac-Moody Lie algebra (see for example \cite{Lusztig6}). Sevenhant and Van den Bergh \cite{Sevenhant-VandenBergh2} took this further and showed that the whole Ringel-Hall algebra can be viewed as the positive part of the quantised enveloping algebra of a Borcherds Lie algebra. These results deepened the connections between quantum groups and representations of quivers, and led to the introduction of Lusztig's canonical basis \cite{Lusztig2, Lusztig3, Lusztig4, Lusztig7}.

Completing the circle, Deng and Xiao showed in \cite{Deng-Xiao2} how the Ringel-Hall algebra could be used to provide a different proof of Kac's Theorem, and actually improve upon Kac's original result, since they show that the dimension vector map from indecomposable representations to positive roots is surjective for any finite field.

The Ringel-Hall algebra construction carries over to any exact hereditary category \cite{Hubery2}, and in particular to the categories of coherent sheaves over smooth projective curves. The case of $\mathbb P^1$ has been extensively studied in \cite{Kapranov, Baumann-Kassel}, and Schiffmann has considered weighted projective lines \cite{Schiffmann3} and elliptic curves \cite{Burban-Schiffmann, Schiffmann5, Schiffmann-Vasserot}, the latter together with Burban and Vasserot. Joyce has also consider Ringel-Hall algebras in the context of configurations of abelian categories \cite{Joyce}.

We also mention work of Reineke \cite{Reineke4, Reineke5, Reineke6, Reineke7, Reineke8} and Reineke and Caldero \cite{Caldero-Reineke1, Caldero-Reineke2} for other interesting occurrences of Ringel-Hall algebras, especially with regard to answering questions in algebraic geometry. Furthermore, there has been some recent work by Caldero and Chapoton relating Ringel-Hall algebras to cluster algebras (see also \cite{Hubery4}). On the other hand, To\"en has shown how to construct a Ringel-Hall algebra from a dg-category \cite{Toen}, a result which has subsequently been extended by Xiao and Xu to more general triangulated categories \cite{Xiao-Xu}.

Our aim in these notes is to present some of this rich theory in the special case of a cyclic quiver. In this case one has a strong connection to the theory of symmetric functions, and we describe this quite thoroughly in the classical case, where the quiver has just a single vertex and a single loop. Our presentation is chosen such that the methods generalise to larger cyclic quivers, and in particular we emphasise the Hopf algebra structure. In the general case we outline a proof that the centre of the Ringel-Hall algebra is isomorphic to the ring of symmetric functions (after extending scalars). This proof is different from Schiffmann's original approach \cite{Schiffmann1}, which relied heavily on some calculations by Leclerc, Thibon and Vasserot \cite{Leclerc-Thibon-Vasserot}. Instead we follow Sevenhant and Van den Bergh \cite{Sevenhant-VandenBergh2}, putting this result in a broader context and avoiding the more involved computations.

The reader might like to consider these notes as a companion to Schiffmann's survey article \cite{Schiffmann4}; the latter is much more advanced and has a much broader scope than these notes, whereas we have tried to fill in some of the gaps. In this spirit we remark that Schiffmann's conjecture is answered by Theorem \ref{centre} (since the map $\Phi_n$ preserves the Hopf pairing), and we finish with Conjecture \ref{Conj} which, if true, would answer the question posed by Schiffmann concerning the (dual) canonical basis elements.

\section{Symmetric Functions}

Symmetric functions play a central role in many areas of mathematics, including the representation theory of the general linear group, combinatorics, analysis and mathematical physics. Here we briefly outline some of the results we shall need in discussing their relationship to Ringel-Hall algebras of cyclic quivers. Our main reference for this section is \cite{Macdonald2}.

\subsection{Partitions}

A partition $\lambda=(\lambda_1,\lambda_2,\ldots,\lambda_l)$ is a finite sequence of positive integers such that $\lambda_i\geq\lambda_{i+1}$ for all $i$. We define $\ell(\lambda):=l$ to be its length, and set
\[ |\lambda|:=\sum_i\lambda_i \quad\textrm{and}\quad m_r(\lambda):=|\{i:\lambda_i=r\}| \quad\textrm{for }r\geq1. \]
It is usual to depict a partition as a Young diagram, where the $i$-th row contains $\lambda_i$ boxes, and the rows are left-justified.

Given $\lambda$, we can reflect its Young diagram in the main diagonal to obtain the Young diagram of another partition, called the dual partition $\lambda'$. We see immediately that $|\lambda'|=|\lambda|$ and that $\lambda'_i = |\{j:\lambda_j\geq i\}|$, so $m_r(\lambda)=\lambda'_r-\lambda'_{r+1}$.

For example, here are the Young diagrams of two conjugate partitions
\begin{center}
\newbox\ASYbox
\newdimen\ASYdimen
\def\ASYbase#1#2{\setbox\ASYbox=\hbox{#1}\ASYdimen=\ht\ASYbox%
\setbox\ASYbox=\hbox{#2}\lower\ASYdimen\box\ASYbox}
\def\ASYalign(#1,#2)(#3,#4)#5#6{\leavevmode%
\setbox\ASYbox=\hbox{#6}%
\setbox\ASYbox\hbox{\ASYdimen=\ht\ASYbox%
\advance\ASYdimen by\dp\ASYbox\kern#3\wd\ASYbox\raise#4\ASYdimen\box\ASYbox}%
\put(#1,#2){%
\wd\ASYbox 0pt\dp\ASYbox 0pt\ht\ASYbox 0pt%
\box\ASYbox%
}}
\def\ASYraw#1{#1}
\setlength{\unitlength}{1pt}
\includegraphics{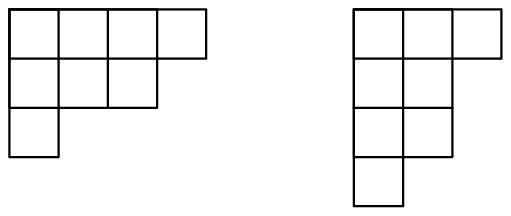}%
\definecolor{ASYcolor}{gray}{0.000000}\color{ASYcolor}
\fontsize{9.962640}{11.955168}\selectfont
\usefont{OT1}{cmr}{m}{n}%
\ASYalign(-114.061737,5.801125)(-0.500000,-0.250000){1.000000 0.000000 0.000000 1.000000}{$(4,3,1)$}
\ASYalign(-21.590462,5.801125)(-0.500000,-0.250000){1.000000 0.000000 0.000000 1.000000}{$(3,2,2,1)$}
\end{center}

The dominance (partial) ordering on partitions of $n$ is given by
\[ \lambda\leq\mu \quad\textrm{if, for all }i,\quad \lambda_1+\cdots+\lambda_i\leq\mu_1+\cdots+\mu_i. \]
It is a nice exercise\footnote{
Suppose $\lambda\leq\mu$. We must have $\ell(\lambda)\geq\ell(\mu)$, so $\lambda'_1\geq\mu'_1$. Now remove the first column of $\lambda$ and place it below the second column, to obtain $\tilde\lambda$ such that $\ell(\tilde\lambda)=\lambda'_1+\lambda'_2$. Do the same to $\mu$, and note that $\tilde\lambda\leq\tilde\mu$.
}
to show that
\[ \lambda\leq\mu \quad\textrm{if and only if}\quad \mu'\leq\lambda'. \]

We also define 
\[ n(\lambda) := \sum_i(i-1)\lambda_i = \sum_i\binom{\lambda_i'}{2}. \]
The equality comes from filling each box in the $i$-th row of the Young diagram for $\lambda$ with the number $i-1$. We then sum these numbers either along rows or along columns. Since $n(\lambda)=\sum_{i<j}\min\{\lambda_i,\lambda_j\}$, we also get
\[ \sum_{i,j}\min\{\lambda_i,\lambda_j\} = \sum_{r,s}m_r(\lambda)m_s(\lambda)\min\{r,s\} = 2n(\lambda)+|\lambda|. \]

Finally, set
\[ z_\lambda := \prod_r\big(m_r(\lambda)!r^{m_r(\lambda)}\big), \]
and note that, if $|\lambda|=n$, then $z_\lambda$ is the size of the centraliser in the symmetric group $\mathfrak S_n$ of any element of cycle type $\lambda$. We have the identities
\[ \sum_{|\lambda|=n}\frac{1}{z_\lambda} = 1 \quad\textrm{and}\quad \sum_{|\lambda|=n}(-1)^{\ell(\lambda)}\frac{1}{z_\lambda} = -\delta_{1n}. \]
The first follows from the Orbit-Stabiliser Theorem, whereas the second follows from the fact that, for $n\geq2$, the alternating group has index 2 in the full symmetric group.

\subsection{The ring of symmetric functions}

Let
\[ \Lambda := \mathbb Q[p_1,p_2,\ldots] \]
be a polynomial ring in countably many variables. This has a $\mathbb Q$-basis indexed by the set of partitions
\[ p_\lambda := \prod_ip_{\lambda_i} = \prod_r p_r^{m_r(\lambda)}, \]
and is naturally $\mathbb N$-graded, where $\deg(p_\lambda):=|\lambda|$.

We make $\Lambda$ into a graded Hopf algebra via
\[ \Delta(p_r):=p_r\otimes1+1\otimes p_r, \quad \varepsilon(p_r):=0 \quad\textrm{and}\quad S(p_r):=-p_r. \]
Thus the generators $p_r$ are primitive elements. Clearly $\Lambda$ is both commutative and cocommutative.

We next define a non-degenerate symmetric bilinear form on $\Lambda$ via
\[ \langle p_\lambda,p_\mu\rangle := \delta_{\lambda\mu}z_\lambda. \]
We observe that this form respects the grading on $\Lambda$. Moreover,
\[ \langle f,gh\rangle = \langle\Delta(f),g\otimes h\rangle, \quad \langle S(f),g\rangle = \langle f,S(g)\rangle, \quad \langle f,1\rangle = \varepsilon(f), \]
where $\langle a\otimes b,c\otimes d\rangle:=\langle a,c\rangle\langle b,d\rangle$. Thus $\langle-,-\rangle$ is a non-degenerate graded Hopf pairing on $\Lambda$.

We summarise this by saying that $\Lambda$ is a self-dual graded Hopf algebra. Note that, since each graded part of $\Lambda$ is finite dimensional, $\Lambda$ is isomorphic (as a graded Hopf algebra) to its graded dual. 

We call $\Lambda$ (equipped with all this extra structure) the ring of symmetric functions.

\subsection{A remark on the Hopf algebra structure}

Let $K$ be a field and $V$ a $K$-vector space, say with basis $\{x_r\}$. Let $S=S(V)$ be the symmetric algebra of $V$, so that $S=K[\{x_r\}]$ is the polynomial ring on the variables $x_r$. Then $S$ is naturally a Hopf algebra via
\[ \Delta(x_r)=x_r\otimes1+1\otimes x_r, \quad S(x_r)=-x_r \quad\textrm{and}\quad \varepsilon(x_r)=0. \]
We can see this either by noting that $V$ is an algebraic group with respect to addition, so its ring of regular functions $S$ is a Hopf algebra, or else that $V$ is an abelian Lie algebra, so its universal enveloping algebra $S$ is a Hopf algebra.

Now, if $\langle-,-\rangle$ is any symmetric bilinear form on $V$ for which the $x_r$ are pairwise orthogonal, say $\langle x_r,x_s\rangle=\delta_{rs}a_r$, then there is a unique extension of this form to a Hopf pairing on $S$, satisfying
\[ \langle x_\lambda,x_\mu \rangle = \delta_{\lambda\mu}\prod_r\big(m_r(\lambda)!a_r^{m_r(\lambda)}\big), \quad\textrm{where } x_\lambda:=\prod_ix_{\lambda_i}. \]
In particular, $\Lambda$ is the symmetric algebra of the $\mathbb Q$-vector space with basis $\{p_r\}$, and we have used the symmetric bilinear form $\langle p_r,p_s\rangle = \delta_{rs}r$.

\subsection{Symmetric Functions}

We shall now describe the relationship between $\Lambda$ and the rings of symmetric polynomials.

Set $R_n:=\mathbb Q[X_1,\ldots,X_n]$. The symmetric group $\mathfrak S_n$ acts on $R_n$ by permuting the $X_i$, and we call the fixed-point ring $S_n:=R_n^{\mathfrak S_n}$ the ring of symmetric polynomials. Clearly both $R_n$ and $S_n$ are $\mathbb N$-graded, where $\deg(X_i):=1$.

The power sum polynomials $p_{r,n}:=X_1^r+\cdots+X_n^r$ are obviously symmetric, and it is a classical result that the $p_{r,n}$ for $1\leq r\leq n$ are algebraically independent and generate $S_n$, so $S_n$ is again a polynomial ring on $n$ generators.

We can make the $R_n$ (and by restriction the $S_n$) into an inverse system using the maps $\rho_n\colon R_n\twoheadrightarrow R_{n-1}$, $X_n\mapsto 0$. Then the graded epimorphisms $\pi_n\colon\Lambda\twoheadrightarrow S_n$ sending $p_r\mapsto p_{r,n}$ are compatible with the $\rho_n$.

\begin{Thm}\label{symm-fns}
We have $\Lambda\cong\varprojlim S_n$ in the category of graded rings.
\end{Thm}

In particular, the $d$-th graded part of $\Lambda$ is the inverse limit of the $d$-th graded parts of the $S_n$. Note that this theorem is only valid if we take the inverse limit in the category of \emph{graded} rings.

This offers an alternative approach to the naturality of the Hopf algebra structure. For, we have isomorphisms
\[ R_n\otimes R_n\xrightarrow{\sim}R_{2n}, \quad X_i\otimes1\mapsto X_i, \quad 1\otimes X_i\mapsto X_{n+i}, \]
and the comultiplication on $\Lambda$ is such that the following diagram commutes
\[ \begin{CD}
\Lambda @>{\Delta}>> \Lambda\otimes\Lambda\\
@VV{\pi_{2n}}V @VV{\pi_n\otimes\pi_n}V\\
R_{2n} @<{\sim}<< R_n\otimes R_n
\end{CD} \]

It is often convenient to express the elements of $\Lambda=\varprojlim S_n$ in terms of the infinite polynomial ring $\varprojlim R_n=\mathbb Q[X_1,X_2,\ldots]$, where the inverse limit is again taken in the category of graded rings. For example, we have $p_r=\sum_iX_i^r$.

\subsection{Special Functions}

There are, of course, many different bases for $\Lambda$. We list below some of the more important ones. We shall often describe elements implicitly by giving their generating function. It is also easy to express the comultiplication in this way, where we extend $\Delta$ to a map $\Lambda\otimes\bQ[T]\to\Lambda\otimes\Lambda\otimes\bQ[T]$ via $\Delta(fT^n)\mapsto\Delta(f)T^n$.

We shall frequently use the following lemma.

\begin{Lem}\label{gen-fns}
Consider homogeneous elements $x_n$ and $y_n$ of degree $n$ such that
\[ \sum_{n\geq0}x_nT^n = \exp\Big(\sum_{n\geq1}\frac{1}{n}y_nT^n\Big). \]
Then
\[ \Delta(x_n) = \sum_{a+b=n}x_a\otimes x_b \quad\textrm{if and only if}\quad \Delta(y_n)=y_n\otimes1+1\otimes y_n. \]
In this case, setting $\xi_n:=\langle x_n,x_n\rangle$ and $\eta_n:=\frac{1}{n}\langle y_n,y_n\rangle$ we similarly have
\[ \sum_{n\geq0}\xi_nT^n = \exp\Big(\sum_{n\geq1}\frac{1}{n}\eta_nT^n\Big). \]
\end{Lem}

\begin{proof}
Set
\[ X(T):=\sum_{n\geq0}x_nT^n, \quad\textrm{and}\quad Y(T):=\sum_{n\geq1}y_nT^{n-1}, \]
and note that $x_0=1$. We can express the relationship between the $x_n$ and $y_n$ in terms of their generating functions as
\[ \dT\log X(T) = Y(T), \quad\textrm{so}\quad X(T)Y(T)=\dT X(T). \]
From these we get
\[ nx_n=\sum_{a=1}^ny_ax_{n-a}, \quad x_n=\sum_{|\lambda|=n}\frac{1}{z_\lambda}y_\lambda, \quad y_n=-n\sum_{|\lambda|=n}\frac{(-1)^{\ell(\lambda)}(\ell(\lambda)-1)!}{\prod_rm_r(\lambda)!}x_\lambda, \]
where as usual $x_\lambda:=\prod_ix_{\lambda_i}$, and analogously for $y_\lambda$.

Using the first equality, we see by induction that
\begin{align*}
\Delta(nx_n-y_n) &= \sum_{a=1}^{n-1}\Delta(y_{n-a}x_a)\\
&= \sum_{a=1}^{n-1}(y_{n-a}\otimes1+1\otimes y_{n-a})(x_a\otimes1+x_{a-1}\otimes x_1+\cdots+1\otimes x_a)\\
&= n(x_n\otimes1+x_{n-1}\otimes x_1+\cdots+1\otimes x_n) - (y_n\otimes1+1\otimes y_n).
\end{align*}
Hence
\[ \Delta(x_n) = \sum_{a+b=n}x_a\otimes x_b \quad\textrm{if and only if}\quad \Delta(y_n)=y_n\otimes1+1\otimes y_n. \]

Now, assuming this, apply $\langle y_n,-\rangle$. Since $y_n$ is primitive, $\langle y_n,fg\rangle=0$ if $f$ and $g$ are both homogeneous of degree at least 1. Therefore
\[ n\langle y_n,x_n\rangle = \langle y_n,y_n\rangle. \]
Now apply $\langle x_n,-\rangle$ to get
\[ n\langle x_n,x_n\rangle = \sum_{a=1}^n\langle\Delta(x_n),y_a\otimes x_{n-a}\rangle = \sum_{a=1}^n\langle x_a,y_a\rangle\langle x_{n-a},x_{n-a}\rangle. \]
Putting these together we get
\[ n\xi_n = \sum_{a=1}^n \eta_a\xi_{n-a}. \]
The result about their generating functions now follows, noting that $\xi_0=1$.
\end{proof}

\subsubsection{Power Sum Functions}

The functions $p_n$ are called the power sum functions. They are characterised up to scalars by being primitive elements:
\[ \Delta(p_n) = p_n\otimes1+1\otimes p_n. \]
The $p_n$ have the generating function
\[ P(T) := \sum_{n\geq1}p_nT^{n-1} = \sum_i\frac{X_i}{1-X_iT}. \]
Since the power sum functions are primitive, we can write
\[ \Delta(P(T)) = P(T)\otimes1+1\otimes P(T). \]
Finally, we recall that
\[ \langle p_\lambda,p_\mu\rangle = \delta_{\lambda\mu}z_\lambda. \]

\subsubsection{Elementary Symmetric Functions}

These are defined via
\[ E(T) = \sum_{n\geq0}e_nT^n := \prod_i(1+X_iT), \quad\textrm{so}\quad e_n = \sum_{i_1<\cdots<i_n}X_{i_1}\cdots X_{i_n}. \]
We observe that
\[ \dT\log E(T) = P(-T), \quad\textrm{so}\quad ne_n = -\sum_{a=1}^n(-1)^ae_{n-a}p_a. \]
Alternatively, we can write
\[ E(T) = \exp\Big(-\sum_{n\geq1}\frac{(-1)^n}{n}p_nT^n\Big), \quad\textrm{so}\quad e_n = (-1)^n\sum_{|\lambda|=n}(-1)^{\ell(\lambda)}\frac{1}{z_\lambda}p_\lambda. \]
It follows from Lemma \ref{gen-fns} that
\[ \Delta(E(T)) = E(T)\otimes E(T), \quad\textrm{or}\quad \Delta(e_n) = \sum_{a+b=n}e_a\otimes e_b \]
and that
\[ \langle e_m,e_n\rangle = \delta_{mn}. \]

\subsubsection{Complete Symmetric Functions}

These are defined via
\[ H(T) = \sum_{n\geq0}h_nT^n := \prod_i(1-X_iT)^{-1}, \quad\textrm{so}\quad h_n = \sum_{i_1\leq\cdots\leq i_n}X_{i_1}\cdots X_{i_n}. \]
We observe that $H(T)E(-T)=1$, so
\[ \sum_{a+b=n}(-1)^ae_ah_b=0 \quad\textrm{for }n\geq1, \]
and giving the analogous statements
\begin{alignat*}{2}
\dT\log H(T) &= P(T), &\quad nh_n &= \sum_{a=1}^nh_{n-a}p_a\\
H(T) &= \exp\Big(\sum_{n\geq1}\frac{1}{n}p_nT^n\Big), &\quad h_n &= \sum_{|\lambda|=n}\frac{1}{z_\lambda}p_\lambda.\\
\Delta(H(T)) &= H(T)\otimes H(T), &\quad \Delta(h_n) &= \sum_{a+b=n}h_a\otimes h_b
\end{alignat*}
and
\[ \langle h_m,h_n\rangle =\delta_{mn}. \]

\subsubsection{Monomial Functions}

To describe the basis of monomial functions, we need a little more notation. Given a finite sequence $\alpha=(\alpha_1,\alpha_2,\ldots)$ of non-negative integers, we can copy the definitions for partitions and set $|\alpha|:=\sum_i\alpha_i$ and $m_r(\alpha):=|\{i:\alpha_i=r\}|$ for $r\geq1$. We write $\alpha\sim\beta$ if $m_r(\alpha)=m_r(\beta)$ for all $r\geq1$. Clearly, given $\alpha$, there is a unique partition $\lambda$ such that $\alpha\sim\lambda$.

Given such a sequence $\alpha$, set $X^\alpha:=\prod_iX_i^{\alpha_i}$, a monomial of degree $|\alpha|$. Then for each partition $\lambda$ we define the monomial function $m_\lambda\in\Lambda$ to be
\[ m_\lambda := \sum_{\alpha\sim\lambda}X^\alpha. \]
The $m_\lambda$ form a basis of $\Lambda$, and
\[ p_n=m_{(n)}, \quad e_n=m_{(1^n)} \quad\textrm{and}\quad h_n=\sum_{|\lambda|=n}m_\lambda. \]

If we set $e_\lambda:=\prod_ie_{\lambda_i}$, then the $e_\lambda$ also form a basis for $\Lambda$. With respect to this basis we can write
\[ m_\lambda = e_{\lambda'} + \sum_{\mu<\lambda}\alpha_{\lambda\mu}e_{\mu'} \quad\textrm{for some integers }a_{\lambda\mu}, \]
where $\lambda'$ again denotes the conjugate partition to $\lambda$.

Similarly we can set $h_\lambda:=\prod_ih_{\lambda_i}$, in which case the $h_\lambda$ form a basis for $\Lambda$ dual to the $m_\lambda$
\[ \langle h_\lambda,m_\mu\rangle = \delta_{\lambda\mu}. \]

\subsubsection{Schur Functions}

The Schur functions play a fundamental role in the representation theories of the symmetric groups and the general linear groups; for example, they correspond to the irreducible characters of $\mathfrak S_n$, and also to the irreducible polynomial representations of $GL_n$.

The Schur functions $s_\lambda$ are characterised by the two properties\footnote{
This is a non-standard description. Usually one replaces property (a) by the equivalent property
\[ s_\lambda=m_\lambda+\sum_{\mu<\lambda}K_{\lambda\mu}m_\mu. \]
The coefficients $K_{\lambda\mu}$ which occur are called the Kostka numbers. See for example \cite{Macdonald3}.
}
\begin{enumerate}
\item[(a)] $s_\lambda = e_{\lambda'}+\sum_{\mu<\lambda}\beta_{\lambda\mu}e_{\mu'}$ for some integers $\beta_{\lambda\mu}$
\item[(b)] $\langle s_\lambda,s_\mu\rangle = \delta_{\lambda\mu}$
\end{enumerate}
and hence the $s_\lambda$ form a basis of $\Lambda$.

More explicitly, we have
\[ s_\lambda = \det\big(e_{\lambda'_i-i+j}\big) = \det\big(h_{\lambda_i-i+j}\big), \]
where the first matrix has size $\lambda_1=\ell(\lambda')$ and the second has size $\ell(\lambda)$. In particular, we always have
\[ s_{(1^n)} = e_n \quad\textrm{and}\quad s_{(n)} = h_n. \]

We compute the first few Schur functions for reference.
\begin{gather*}
s_{(1)}=e_1\\
s_{(1^2)}=e_2, \quad s_{(2)}=e_1^2-e_2=h_2\\
s_{(1^3)}=e_3, \quad s_{(12)}=e_1e_2-e_3, \quad s_{(3)}=e_1^3-2e_1e_2+e_3=h_3\\
s_{(1^4)}=e_4, \quad s_{(1^22)}=e_1e_3-e_4, \quad s_{(2^2)}=e_2^2-e_1e_3,\\
s_{(13)}=e_1^2e_2-e_2^2-e_1e_3+e_4, \quad s_{(4)}=e_1^4-3e_1^2e_2+e_2^2+2e_1e_3-e_4=h_4.
\end{gather*}

If we write
\[ s_\lambda s_\mu = \sum_\xi c_{\lambda\mu}^\xi s_\xi, \]
then the coefficients $c_{\lambda\mu}^\xi$ are called the Littlewood-Richardson coefficients.

\subsection{An Important Generalisation}

We shall see in the next chapter that, when studying the representation theory of finite abelian $p$-groups, or nilpotent modules for the polynomial ring $\mathbb F_q[X]$, the extension $\Lambda[t]:=\Lambda\otimes\mathbb Q[t]$ of $\Lambda$ arises naturally. %This has since been found also to play a crucial role in understanding the modular representation theory of the general linear group.
In this setting it is useful to redefine the symmetric bilinear form to be
\[ \langle p_\lambda,p_\mu\rangle_t := \delta_{\lambda\mu}z_\lambda(t), \]
where
\[ z_\lambda(t) := \prod_r\bigg(m_r(\lambda)!\Big(\frac{r}{1-t^r}\Big)^{m_r(\lambda)}\bigg) = z_\lambda\prod_r(1-t^r)^{-m_r(\lambda)} \in \mathbb Q(t). \]
This is a Hopf pairing on $\Lambda[t]$, for the same reasons as before, and it is clear that specialising to $t=0$ recovers the original form.

We set
\[ \phi_n(t)=(1-t)(1-t^2)\cdots(1-t^n) \quad\textrm{and}\quad b_\lambda(t):=\prod_r\phi_{m_r(\lambda)}(t). \]

Then\footnote{
By Lemma \ref{gen-fns} we have
\[ \sum_{n\geq0}\langle e_n,e_n\rangle_tT^n = \mathrm{exp}\Big(\sum_{n\geq1}\frac{1}{n(1-t^n)}T^n\Big). \]
We can rewrite this as $\prod_{n\geq0}(1-t^nT)^{-1}$, and expanding the product we obtain
\[ \prod_{n\geq0}(1-t^nT)^{-1} = (1-T)^{-1}\sum_\lambda t^{|\lambda|}T^{\ell(\lambda)} = (1-T)^{-1}\sum_\lambda t^{|\lambda|}T^{\lambda_1} = \sum_{n\geq0}\phi_n(t)^{-1}T^n. \]
}
\[ \langle e_m,e_n\rangle_t = \langle h_m,h_n\rangle_t = \delta_{mn}\phi_n(t)^{-1}. \]

\begin{comment}
\[ \prod_{s\geq0}(1-t^sT)^{-1} = \sum_{n\geq0}T^n\sum_{\lambda,\ \lambda_1\leq n}t^{|\lambda|} \]
whereas
\[ \sum_{r\geq0}\phi_r(t)^{-1}T^r = \sum_{n\geq0}T^n\sum_{\lambda,\ \ell(\lambda)\leq n}t^{|\lambda|}. \]
Taking dual partitions, we see that these two expressions are equal.
\end{comment}

\subsection{Dual Schur Functions}

With respect to this new bilinear form, the Schur functions no longer give an orthonormal basis. We therefore introduce the dual Schur functions $S_\lambda(t)$ such that $\langle S_\lambda(t),s_\mu\rangle_t=\delta_{\lambda\mu}$. The $S_\lambda(t)$ can be given explicitly via
\[ S_\lambda(t) = \det\big(c_{\lambda_i-i+j}(t)\big), \quad\textrm{where as usual}\quad c_\lambda(t)=\prod_ic_{\lambda_i}(t). \]
In particular, we have
\[ S_{(n)}(t) = c_n(t). \]

\subsubsection{Cyclic Symmetric Functions}

We generalise the complete symmetric functions by\footnote{
In Macdonald's book, they are denoted $q_n(\underline X;t)$, but $q$ seems an unfortunate choice since this is used elsewhere as another variable.}
\[ C(T) := 1+\sum_{n\geq1}c_n(t)T^n = \exp\Big(\sum_{n\geq1}\frac{1-t^n}{n}p_nT^n\Big), \]
so, setting $c_0(t):=1$ for convenience,
\[ c_n(t) = \sum_{|\lambda|=n}\frac{1}{z_\lambda(t)}p_\lambda \quad\textrm{and}\quad nc_n(t) = \sum_{a=1}^n(1-t^a)p_ac_{n-a}(t). \]
We can also express the generating function $C(T)$ in terms of $E(T)$ and $H(T)$:
\[ C(T) = H(T)/H(tT) = E(-tT)/E(-T). \]
This gives
\[ (1-t^n)h_n = \sum_{a=1}^nt^{n-a}h_{n-a}c_a(t) \quad\textrm{and}\quad (t^n-1)e_n = \sum_{a=1}^n(-1)^ae_{n-a}c_a(t). \]
We note that specialising to $t=0$ recovers the complete symmetric functions, $c_n(0)=h_n$.

The functions $c_n(t)$ do not seem to have a name, so we shall refer to them as `cyclic' functions based on their role in the Hall algebra.

We have
\[ \Delta(C(T)) = C(T)\otimes C(T), \quad \Delta(c_n(t)) = \sum_{a+b=n}c_a(t)\otimes c_b(t) \]
and, using Lemma \ref{gen-fns},
\[ \sum_{n\geq0}\langle c_n(t),c_n(t)\rangle_tT^n = \exp\Big(\sum_{n\geq1}\frac{1-t^n}{n}T^n\Big) = \frac{1-tT}{1-T} = 1+(1-t)\sum_{n\geq1}T^n. \]
Thus
\[ \langle c_n(t),c_n(t)\rangle_t = (1-t) \quad\textrm{for }n\geq1. \]
Setting $c_\lambda(t):=\prod_ic_{\lambda_i}(t)$, then the $c_\lambda(t)$ form a basis for $\Lambda[t]$ dual to the basis of monomial functions
\[ \langle c_\lambda(t),m_\mu\rangle_t = \delta_{\lambda\mu}. \]

\subsubsection{Hall-Littlewood Functions}

One can also generalise the Schur functions $s_\lambda$ to obtain the Hall-Littlewood symmetric functions $P_\lambda(t)$. These are characterised by\footnote{
We have again replaced the standard description in terms of the monomial functions by an equivalent one involving the elementary symmetric functions.
}
\begin{enumerate}
\item[(a)] $P_\lambda(t) = e_{\lambda'} + \sum_{\mu<\lambda}\beta_{\lambda\mu}(t)e_{\mu'}$ for some integer polynomials $\beta_{\lambda\mu}(t)$
\item[(b)] $\langle P_\lambda(t),P_\mu(t)\rangle_t = \delta_{\lambda\mu}b_\lambda(t)^{-1}$
\end{enumerate}
so the $P_\lambda(t)$ form a basis for $\Lambda[t]$.

Clearly $P_\lambda(0)=s_\lambda$, but we also have that $P_\lambda(1)=m_\lambda$, so the Hall-Littlewood functions can be thought of as providing a transition between the Schur functions and the monomial functions. In fact, we have 
\[ s_\lambda = P_\lambda(t) + \sum_{\mu<\lambda}K_{\lambda\mu}(t)P_\mu(t), \]
and the coefficients $K_{\lambda\mu}(t)$ are integer polynomials, called the Kostka-Foulkes polynomials. Note that, since $P_\lambda(0)=s_\lambda$, we must have $K_{\lambda\mu}(t)\in t\mathbb Z[t]$.

\begin{comment}
The coefficients $K_{\lambda\mu}(t)$ are called the Kostka-Foulkes polynomials. It was proven by Lascoux and Sch\"utzenberger [LS] that 
\[ K_{\lambda\mu}(t) = \sum_Tt^{c(T)}, \]
where the sum is taken over all semistandard tableaux of shape $\lambda$ and weight $\mu$, and $c(T)$ is the so-called charge of $T$. In particular, $K_{\lambda\mu}(1)=K_{\lambda\mu}$. Moreover, if $\lambda\geq\mu$, then $K_{\lambda\mu}(t)$ is monic of degree $n(\mu)-n(\lambda)$.

The coefficients $\kappa_{\lambda\mu}(t)$ are related to the Kostka numbers via
\[ K_{\lambda\mu} = \sum_{\lambda\geq\nu\geq\mu}K_{\lambda\nu}(t)\kappa_{\nu\mu}(t), \]
and these satisfy $\kappa_{\lambda\mu}(0)=K_{\lambda\mu}$ and $\kappa_{\lambda\mu}(1)=\delta_{\lambda\mu}$. We also have the formula\footnote{
Given a semistandard tableau $T$ of shape $\lambda$ and weight $\mu$, let $\lambda^{(i)}$ denote the partition given by all boxes with label at most $i$. Also, set $J_i$ to be those $j$ such that column $j+1$ contains a box with label $i+1$, but column $j$ does not. Then
\[ \psi_T(t) := \prod_i\psi_{T,i}(t) \quad\textrm{and}\quad \psi_{T,i}(t) := \prod_{j\in J_i}\big(1-t^{m_j(\lambda^{(i)})}\big). \]
}
\[ \kappa_{\lambda\mu}(t) = \sum_T\psi_T(t), \]
where the sum is again over all semistandard tableaux of shape $\lambda$ and weight $\mu$.
\end{comment}

Finally, we state the relations
\[ e_r = P_{(1^r)}(t), \quad c_r(t) = (1-t)P_{(r)}(t), \quad h_r = \sum_{|\lambda|=r}t^{n(\lambda)}P_\lambda(t) \]
and
\[ p_r = \sum_{|\lambda|=r}(1-t^{-1})(1-t^{-2})\cdots(1-t^{1-\ell(\lambda)})t^{n(\lambda)}P_\lambda(t). \]
We will prove these using the classical Hall algebra.

\subsection{Integral Bases}

We observe from the formulae
\[ \sum_{a+b=n}(-1)^ae_ah_b = 0, \quad m_\lambda = e_{\lambda'}+\sum_{\mu<\lambda}\alpha_{\lambda\mu}e_{\mu'} \quad\textrm{and}\quad s_\lambda = e_{\lambda'}+\sum_{\mu<\lambda}\beta_{\lambda\mu}e_{\mu'}, \]
where $\alpha_{\lambda\mu},\beta_{\lambda\mu}\in\mathbb Z$, that the following subrings are all equal
\[ \mathbb Z[e_1,e_2,\ldots] = \mathbb Z[h_1,h_2,\ldots] = \mathbb Z[\{m_\lambda\}] = \mathbb Z[\{s_\lambda\}]. \]
If we denote this subring by ${}_{\mathbb Z}\Lambda$, then we can strengthen Theorem \ref{symm-fns} to give

\begin{Thm}
\[ {}_{\mathbb Z}\Lambda \cong \varprojlim \mathbb Z[X_1,\ldots,X_n]^{\mathfrak S_n}. \]
\end{Thm}

On the other hand, since
\[ ne_n=-\sum_{a=1}^n(-1)^ae_{n-a}p_a, \]
the $p_\lambda$ do not form a basis for ${}_{\mathbb Z}\Lambda$.

We can similarly study the subring ${}_{\mathbb Z}\Lambda[t]$ of $\Lambda[t]$. Then the formula
\[ P_\lambda(t) = e_{\lambda'} + \sum_{\mu<\lambda}\beta_{\lambda\mu}(t)e_{\mu'}, \quad \beta_{\lambda\mu}(t)\in\mathbb Z[t], \]
shows that the $P_\lambda(t)$ form a basis for ${}_{\mathbb Z}\Lambda[t]$.

However, since
\[ (t^n-1)e_n = \sum_{a=1}^n(-1)^ae_{n-a}c_a(t), \]
we see that the $c_n(t)$ do not even generate $\Lambda[t]$; one would need to invert each polynomial of the form $t^n-1$.

Similarly, using the description of the dual Schur functions, we see that they also do not form a basis of $\Lambda[t]$.

\section{Ringel-Hall Algebras}

We now review the theory of Ringel-Hall algebras, based on the work of Ringel and Green \cite{Ringel1, Green1}.

Let $k$ be a field and let $\mathcal A$ be an abelian (or, more generally, exact) $k$-linear category which
\begin{itemize}
\item is skeletally small, so the isomorphism classes of objects form a set;
\item is hereditary, so that $\Ext^2(-,-)=0$;
\item has finite dimensional hom and ext spaces;
\item has split idempotents, so idempotent endomorphisms induce direct sum decompositions. (This is automatic if $\mathcal A$ is abelian.)
\end{itemize}
The last two conditions imply that $\End(A)$ is a finite-dimensional algebra, which is local precisely when $A$ is indecomposable. Thus $\mathcal A$ is a Krull-Schmidt category, so every object is isomorphic to a direct sum of indecomposable objects in an essentially unique way.

We define the Euler characteristic of $\mathcal A$ to be
\[ \langle M,N\rangle := \dim\Hom(M,N) - \dim\Ext^1(M,N). \]
Since $\mathcal A$ is hereditary, this descends to a bilinear map on the Grothendieck group $K_0(\mathcal A)$ of $\mathcal A$. We shall also need its symmetrisation
\[ (M,N) := \langle M,N\rangle + \langle M,N\rangle. \]
Let
\[ \mathrm{rad}(-,-) = \{\alpha:(\alpha,\beta)=0\textrm{ for all }\beta\} \]
denote the radical of the symmetric bilinear form on $K_0(\mathcal A)$, and set\footnote{
It would be interesting to determine which abelian groups admit such a non-degenerate integer-valued symmetric bilinear form. It is clear that such a group is torsion-free, and easy to show that every finite rank subgroup is free. Thus one can use Pontryagin's Theorem to deduce that every countable subgroup is free (see for example \cite{Fuchs}). On the other hand, if $(x,x)=0$ implies $x=0$, so the form is a Yamabe function \cite{Yamabe}, then the whole group is free. For, if it has finite rank, then it is free \cite{Yamabe}, whereas if the rank is at least 5, then we may assume the form is positive definite (c.f. \cite{Neumann-Neumann}), in which case it is free by \cite{Steprans}.
}
\[ \overline K_0(\mathcal A):=K_0(\mathcal A)/\mathrm{rad}(-,-). \]
In algebraic geometry, this quotient is commonly called the numerical Grothendieck group.

\subsection{Hall numbers}

Given objects $M,N,X$ we define
\[ \mathcal E_{MN}^X := \{(f,g):0\to N\xrightarrow{f}X\xrightarrow{g}M\to 0\textrm{ exact}\}. \]
We observe that $\Aut X$ acts on $\mathcal E_{MN}^X$ via $\alpha_X\cdot(f,g):=(\alpha_X f,g\alpha_X^{-1})$. The map $\theta\mapsto 1+f\theta g$ induces an isomorphism between $\Hom(M,N)$ and the stabiliser of $(f,g)$, and the quotient $\mathcal E_{MN}^X/\Aut X$ equals
\[ \Ext^1(M,N)_X = \{\textrm{extension classes having middle term isomorphic to $X$}\}. \]

On the other hand, $\Aut M\times\Aut N$ acts freely on $\mathcal E_{MN}^X$ via $(\alpha_M,\alpha_N)\cdot(f,g):=(f\alpha_N^{-1},\alpha_M g)$, and we define
\[ \mathcal F_{MN}^X := \frac{\mathcal E_{MN}^X}{\Aut M\times\Aut N} \]
to be the quotient. In the special case when $\mathcal A=\modcat R$ is the category of finite dimensional $R$-modules for some $k$-algebra $R$, then the map $(f,g)\mapsto\mathrm{Im}(f)$ yields the alternative defintion
\[ \mathcal F_{MN}^X = \{U\leq X:U\cong N,\ X/U\cong M\}. \]
This can then be iterated to give
\[ \mathcal F_{M_1\cdots M_n}^X := \{ 0=U_n\leq\cdots\leq U_1\leq U_0=X:U_{i-1}/U_i\cong M_i\}, \]
which is the set of filtrations of $X$ with subquotients $(M_1,\ldots,M_n)$ ordered from the top down.

Taking the union over all possible cokernels we obtain
\[ \coprod_{[M]}\mathcal F_{MN}^X = \frac{\mathrm{Inj}(N,X)}{\Aut(N)}, \]
where $\mathrm{Inj}(N,X)$ is the set of all (admissible) monomorphisms from $N$ to $X$. A dual result obviously holds if we take the union over all possible kernels.

Now suppose that $k$ is a finite field. Then all the sets we have defined so far are finite, so we may consider their cardinalities. We define
\[ a_X:=|\Aut X|, \quad E_{MN}^X:=|\mathcal E_{MN}^X| \quad\textrm{and}\quad F_{MN}^X:=|\mathcal F_{MN}^X|. \]
The $F_{MN}^X$ are called Hall numbers. Note that
\[ E_{MN}^X = \frac{|\Ext^1(M,N)_X|a_X}{|\Hom(M,N)|} \quad\textrm{and}\quad F_{MN}^X = \frac{E_{MN}^X}{a_Ma_N} = \frac{|\Ext^1(M,N)_X|}{|\Hom(M,N)|}\frac{a_X}{a_Ma_N}. \]
The latter is commonly referred to as Riedtmann's Formula \cite{Riedtmann}.

\subsection{The Ringel-Hall Algebra}

We use the numbers $F_{MN}^X$ as structure constants to define the Ringel-Hall algebra $\HA(\mathcal A)$. Let $v_k\in\mathbb R$ be the positive square-root of $|k|$ and let $\bQ_k:=\bQ(v_k)\subset\mathbb R$. Then $\HA(\mathcal A)$ is the $\bQ_k$-algebra with basis the isomorphism classes of objects in $\mathcal A$ and multiplication
\[ u_Mu_N:=v_k^{\langle M,N\rangle}\sum_{[X]}F_{MN}^Xu_X. \]
Note that the sum is necessarily finite, since $\Ext^1(M,N)$ is a finite set.

\begin{Thm}[Ringel]
$\HA(\mathcal A)$ is an associative algebra with unit $[0]$. Moreover, it is naturally graded by $K_0(\mathcal A)$.
\end{Thm}

\begin{proof}
Given $L$, $M$, $N$ and $X$, the pull-back/push-out constructions
\[ \begin{CD}
&& 0 && 0\\
&& @VVV @VVV\\
&& N @= N\\
&& @VVV @VVV\\
0 @>>> B @>>> X @>>> L @>>> 0\\
&& @VVV @VVV @|\\
0 @>>> M @>>> A @>>> L @>>> 0\\
&& @VVV @VVV\\
&& 0 && 0
\end{CD} \]
induce bijections
\[ \coprod_{[A]}\frac{\mathcal E_{LM}^A\times\mathcal E_{AN}^X}{\Aut A} \longleftrightarrow \coprod_{[B]}\frac{\mathcal E_{LB}^X\times\mathcal E_{MN}^B}{\Aut B}, \]
where the automorphism groups act diagonally, and hence freely. This, together with $\langle A,-\rangle=\langle L,-\rangle+\langle M,-\rangle$, yields the associativity law.

Since $F_{MN}^X\neq0$ only if there exists a short exact sequence $0\to N\to X\to M\to 0$, it follows that $\HA(\mathcal A)$ is graded by $K_0(\mathcal A)$.
\end{proof}

Dually, we can endow $\HA(\mathcal A)$ with the structure of a coalgebra. Define
\[ \Delta(u_X) := \sum_{[M],[N]} v_k^{\langle M,N\rangle} \frac{E_{MN}^X}{a_X}u_M\otimes u_N. \]
Since
\[ \Delta(a_Xu_X) = \sum_{[M],[N]} v_k^{\langle M,N\rangle} F_{MN}^X(a_Mu_M)\otimes(a_Nu_N), \]
we see that the comultiplication is in essence dual to the multiplication.

N.B. In order for this definition to make sense, we need that each object $X$ has only finitely many subobjects (that is, $\mathcal A$ is finitely well-powered). This is clearly satisfied for $\mathcal A=\modcat R$.

\begin{Thm}
$\HA(\mathcal A)$ is a graded coassociative coalgebra with counit $\varepsilon([X])=\delta_{[X][0]}$.
\end{Thm}

\begin{proof}
This follows from the same formula that proves associativity.
\end{proof}

Finally, Green used the heredity property to show that $\HA(\mathcal A)$ is a twisted bialgebra. We can construct an honest bialgebra by adjoining the group algebra of $\overline K_0(\mathcal A)$.

Recall that $\bQ_k[\overline K_0(\mathcal A)]$ is a Hopf algebra with basis $K_\alpha$ for $\alpha\in\overline K_0(\mathcal A)$ such that
\[ K_\alpha K_\beta := K_{\alpha+\beta}, \quad \Delta(K_\alpha) := K_\alpha\otimes K_\alpha, \quad \varepsilon(K_\alpha) := 1, \quad S(K_\alpha) := K_{-\alpha}. \]

We define $H(\mathcal A)$ to be a bialgebra such that the natural embedding of $\HA(\mathcal A)$ into $H(\mathcal A)$ is an algebra homomorphism, and the natural embedding of $\bQ_k[\overline K_0(\mathcal A)]$ into $H(\mathcal A)$ is a bialgebra homomorphism.

We do this by first defining $H(\mathcal A)$ to be the smash product\footnote{
See for example \cite{DNR}.
}
$\HA(\mathcal A)\#\mathbb Q_k[\overline K_0(\mathcal A)]$, where we make $\HA(\mathcal A)$ into a $\mathbb Q_k[\overline K_0(\mathcal A)]$-module algebra via
\[ K_\alpha\cdot u_X := v_k^{(\alpha,[X])}u_X. \]
As a vector space, $H(\mathcal A)=\HA(\mathcal A)\otimes\mathbb Q_k[\overline K_0(\mathcal A)]$, and since the $K_\alpha$ are group-like, we just have
\[ K_\alpha u_X=v^{(\alpha,[X])}u_XK_\alpha. \]
Green's result then implies that $H(\mathcal A)$ is a bialgebra, where
\[ \Delta([X]) := \sum_{[M],[N]} v_k^{\langle M,N\rangle}\frac{E_{MN}^X}{a_X}[M]K_N\otimes[N] \quad\textrm{and}\quad \varepsilon([X]) := \delta_{[X]0}. \]
We extend the grading by letting each $K_\alpha$ have degree $0$.

We can also define an antipode on $H(\mathcal A)$, as done by Xiao in \cite{Xiao}.

\begin{Thm}[Green, Xiao]
$H(\mathcal A)$ is a self-dual graded Hopf algebra. Given a linear map $\dim\colon K_0(\mathcal A)\to\mathbb Z$ we have the non-degenerate Hopf pairing
\[ \langle [M]K_\alpha,[N]K_\beta\rangle := \delta_{[M][N]}\frac{v_k^{(\alpha,\beta)+2\dim M}}{a_M}. \]
\end{Thm}

If $\mathcal A=\modcat R$, then one usually takes the linear functional $\dim$ to be the dimension as a $k$-vector space.

\section{Cyclic Quivers, I}

Let $C_1$ be the quiver with a single vertex $1$ and a single loop $a$.
\begin{center}
\newbox\ASYbox
\newdimen\ASYdimen
\def\ASYbase#1#2{\setbox\ASYbox=\hbox{#1}\ASYdimen=\ht\ASYbox%
\setbox\ASYbox=\hbox{#2}\lower\ASYdimen\box\ASYbox}
\def\ASYalign(#1,#2)(#3,#4)#5#6{\leavevmode%
\setbox\ASYbox=\hbox{#6}%
\setbox\ASYbox\hbox{\ASYdimen=\ht\ASYbox%
\advance\ASYdimen by\dp\ASYbox\kern#3\wd\ASYbox\raise#4\ASYdimen\box\ASYbox}%
\put(#1,#2){%
\wd\ASYbox 0pt\dp\ASYbox 0pt\ht\ASYbox 0pt%
\box\ASYbox%
}}
\def\ASYraw#1{#1}
\setlength{\unitlength}{1pt}
\includegraphics{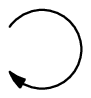}%
\definecolor{ASYcolor}{gray}{0.000000}\color{ASYcolor}
\fontsize{9.962640}{11.955168}\selectfont
\usefont{OT1}{cmr}{m}{n}%
\ASYalign(-60.349470,11.632017)(-0.500000,-0.320001){1.000000 0.000000 0.000000 1.000000}{$C_1\colon$}
\ASYalign(-31.896770,11.632017)(-0.500000,-0.500000){1.000000 0.000000 0.000000 1.000000}{$1$}
\ASYalign(-3.444070,11.632017)(-0.500000,-0.500000){1.000000 0.000000 0.000000 1.000000}{$a$}
\end{center}
For a field $k$ let $\Cat_1(k)$ be the category of $k$-representations of $C_1$; i.e. the category of functors from $C_1$ to finite dimensional $k$-vector spaces. Thus a representation $M$ is given by a finite dimensional vector space $M(1)$ together with an endomorphism $M(a)$.

Equivalently, we can view $\Cat_1(k)$ as the category of finite dimensional $k[T]$-modules. For, such a module is determined by a vector space together with an endomorphism describing how $T$ acts.

Since $k[T]$ is a principal ideal domain, every finite dimensional module can be written as
\[ M \cong \bigoplus_{p,r} \big(k[T]/p^r\big)^{m_r(p)}, \quad\textrm{where } p\in k[T]\textrm{ is monic irreducible and } r\in\mathbb N. \]
In terms of matrices, this corresponds to a rational normal form, generalising the Jordan normal form over algebraically closed fields.

More restrictively, we define $\Cat_1^0(k)$ to be the full subcategory of nilpotent modules. Thus we only allow the irreducible polynomial $p=T$, and hence the isomorphism classes are indexed by partitions (giving the block sizes in the Jordan normal form). As such, this indexing set is independent of the field. We write
\[ M_\lambda := \bigoplus_r \big(k[T]/T^r\big)^{m_r(\lambda)} = \bigoplus_i k[T]/T^{\lambda_i}. \]
We observe that $\Cat_1^0(k)$ is a uniserial length category\footnote{
The lattice of submodules of an indecomposable module is a finite chain.
}
satisfying our previous conditions. Also, $K_0\cong\mathbb Z$ via $[M]\mapsto\dim M$, and $\overline K_0=0$ since the Euler characteristic is identically zero.

Now let $k$ be a finite field. Then the Ringel-Hall algebra $H_1(k):=H(\Cat_1^0(k))$ has basis $u_\lambda:=[M_\lambda]$ and is $\mathbb N$-graded via $\deg(u_\lambda)=|\lambda|$.

\subsection{Examples of Hall numbers}

We now compute the Hall numbers in some easy cases. We shall simplify notation slightly and just write $F_{\lambda\mu}^\xi$ instead of $F_{M_\lambda M_\mu}^{M_\xi}$. This preempts the next section where we prove that the Hall numbers are given by specialising certain Hall polynomials. We set $q:=|k|$.

\begin{enumerate}
\item Let $r=a+b$. Since the category $\Cat_1^0(k)$ is uniserial we have
\[ F_{(a)(b)}^{(r)} = 1. \]
\item For $a\geq2$ and any $b$ we have
\[ u_{(a)}u_{(1^b)} = u_{(1^{b-1} a+1)} + q^bu_{(1^b a)}. \]
For, consider a short exact sequence
\[ 0 \to M_{(1^b)} \to X \to M_{(a)}\to 0. \]
We immediately see that $\mathrm{soc}(X)$ has dimension either $b$ or $b+1$. Similarly, $X$ has Loewy length either $a$ or $a+1$. If $\mathrm{soc}(X)\cong M_{(1^b)}$, then this submodule is uniquely determined, and since the cokernel is isomorphic to $M_{(a)}$, we must have $X\cong M_{(1^{b-1}a+1)}$ with corresponding Hall number 1. On the other hand, if $\mathrm{soc}(X)\cong M_{(1^{b+1})}$, then since $X$ has Loewy length $a$ we have $X\cong M_{(1^ba)}$ and the sequence must be split. In this case $\Ext^1(M_{(a)},M_{(1^b)})_{M_{(1^ba)}}=0$, so we can use Riedtmann's Formula to calculate that $F_{(a)(1^b)}^{(1^ba)}=|\Hom(M_{(1^b)},M_{(a)})|=q^b$.
\item Let $r=a+b$. Then $\mathcal F_{(1^a)(1^b)}^{(1^r)}$ is isomorphic to the Grassmannian $\mathrm{Gr}\binom{r}{a,b}$. For, given the semisimple module $M_{(1^r)}$, choosing a submodule isomorphic to $M_{(1^b)}$ is equivalent to choosing a $b$-dimensional subspace of an $r$-dimensional vector space, and for each such choice, the cokernel will necessarily be semisimple, hence isomorphic to $M_{(1^a)}$. Thus
\[ F_{(1^a)(1^b)}^{(1^r)} = q^{ab}\begin{bmatrix}r\\a\end{bmatrix}_{q^{-1}} \quad\textrm{where}\quad \begin{bmatrix}r\\a\end{bmatrix}_t:=\frac{\phi_r(t)}{\phi_a(t)\phi_b(t)}. \]
\item More generally, we have the formula
\[ F_{\lambda(1^m)}^\xi = q^{n(\xi)-n(\lambda)-n(1^m)}\prod_{i\geq1}\begin{bmatrix}\xi'_i-\xi'_{i+1}\\\xi'_i-\lambda'_i\end{bmatrix}_{q^{-1}}. \]
\item We next prove that
\[ F_{\lambda\mu}^\xi\neq0 \quad\textrm{implies}\quad \lambda\cup\mu\leq\xi\leq\lambda+\mu, \]
where $\lambda\cup\mu$ is the partition formed by concatentating the parts of $\lambda$ and $\mu$, so $m_r(\lambda\cup\mu)=m_r(\lambda)+m_r(\mu)$, and $\lambda+\mu$ is formed by adding the corresponding parts of $\lambda$ and $\mu$, so $(\lambda+\mu)_i=\lambda_i+\mu_i$. Note that $(\lambda\cup\mu)'=\lambda'+\mu'$, so these concepts are dual to one another.

To see this, suppose we have a short exact sequence
\[ 0\to M_\mu\xrightarrow{\iota} M_\xi\xrightarrow{\pi} M_\lambda\to 0. \]
We first consider the socle series for $M_\xi$ compared with the socle series of $M_\lambda\oplus M_\mu$. We have
\begin{align*}
\dim\mathrm{soc}^i(M_\xi) &= \xi'_1+\cdots+\xi'_i\\
\dim\mathrm{soc}^i(M_\lambda\oplus M_\mu) &= \lambda'_1+\mu'_1+\cdots+\lambda'_i+\mu'_i.
\end{align*}
Since $\pi\mathrm{soc}^i(M_\xi)\subset\mathrm{soc}^i(M_\lambda)$ and $\iota^{-1}\mathrm{soc}^i(M_\xi)\subset\mathrm{soc}^i(M_\mu)$, we must have $\xi'\leq(\lambda\cup\mu)'$, or equivalently $\xi\geq\lambda\cup\mu$.

On the other hand, define
\[ M_\xi^{(\leq i)}:=\bigoplus_{j\leq i}M_{\xi_j} \]
to be the sum of the $i$ largest indecomposable summands of $M_\xi$. Then $\pi\big(M_\xi^{(\leq i)}\big)$ and $\iota^{-1}\big(M_\xi^{(\leq i)}\big)$ each have at most $i$ summands, we must have $\xi_1+\cdots+\xi_i\leq\lambda_1+\mu_1+\cdots+\lambda_i+\mu_i$, so that $\xi\leq\lambda+\mu$.
\item Finally, we note that
\[ F_{\lambda\mu}^\xi=F_{\mu\lambda}^\xi, \]
so that the Ringel-Hall algebra is both commutative and cocommutative (since it is self-dual).

For this we observe that there is a natural duality on the category $\Cat_1^0(k)$ given by $D=\Hom_k(-,k)$, and that $D(M_\lambda)\cong M_\lambda$.
\end{enumerate}

\subsection{Hall Polynomials}\label{HallPolys}

We now show that, for each triple $(\lambda,\mu,\xi)$, there exist polynomials $F_{\lambda\mu}^\xi(T)\in\mathbb Z[T]$ such that, for any finite field $k$ and any objects $M_\lambda,M_\mu,M_\xi\in\Cat_1^0(k)$ of types $\lambda,\mu\xi$ respectively,
\[ F_{M_\lambda M_\mu}^{M_\xi} = F_{\lambda\mu}^\xi(|k|). \]
The polynomials $F_{\lambda\mu}^\xi$ are called Hall polynomials, and allow one to form a generic Ringel-Hall algebra.

\begin{Thm}\label{Thm1}
There exist
\begin{enumerate}
\item integers $d(\lambda,\mu)$ such that, over any field $k$,
\[ d(\lambda,\mu) = \dim\Hom(M_\lambda,M_\mu). \]
In fact,
\[ d(\lambda,\mu) = \sum_{i,j}\min\{\lambda_i,\mu_j\}, \quad \textrm{so}\quad d(\lambda,\lambda)=2n(\lambda)+|\lambda|. \]
\item monic integer polynomials $a_\lambda$ such that, over any finite field $k$,
\[ a_\lambda(|k|) = a_{M_\lambda} = |\Aut(M_\lambda)|. \]
In fact,
\[ a_\lambda = T^{2n(\lambda)+|\lambda|}b_\lambda(T^{-1}). \]
\item integer polynomials $F_{\lambda\mu}^\xi$ such that, over any finite field $k$,
\[ F_{\lambda\mu}^\xi(|k|) = F_{M_\lambda M_\mu}^{M_\xi}. \]
Moreover,
\[ F_{\lambda\mu}^\xi = c_{\lambda\mu}^\xi T^{n(\xi)-n(\lambda)-n(\mu)} + \textrm{lower degree terms}, \]
where $c_{\lambda\mu}^\xi$ is the Hall-Littlewood coefficient.
\end{enumerate}
\end{Thm}

The first two statements are easy to prove. For the third, there are several approaches.

In \cite{Macdonald2}, Macdonald proves this using the Littlewood-Richardson rule for computing the coefficients $c_{\lambda\mu}^\xi$. In particular, he first shows how each short exact sequence determines an LR-sequence, so one can decompose the Hall number as $F_{\lambda\mu}^\xi=\sum_SF_S$ corresponding to the possible LR-sequences. Finally he proves that each $F_S$ is given by a universal polynomial.

Alternatively, as detailed in \cite{Schiffmann6}, one can use Example (4) above to prove polynomiality. By iteration, using Example (5), one sees that there exist integer polynomials $f_\lambda^\xi$ such that
\[ u_{(1^{\lambda'_r})}\cdots u_{(1^{\lambda'_1})} = \sum_{\mu\leq\lambda} f_\lambda^\mu(q) u_\mu. \]
Note also that $f_\lambda^\lambda=1$ since the corresponding filtration of $M_\lambda$ is just the socle series. Inverting this shows that any $u_\lambda$ can be expressed as a sum of products of the $u_{(1^r)}$ with coefficients given by integer polynomials. The existence of Hall polynomials follows quickly.

In an appendix to \cite{Macdonald2} Zelevinsky shows how this approach can be taken further to give the full statement. Using Hall's Theorem below, which only requires the existence of Hall polynomials, one sees that
\[ P_\lambda(t)P_\mu(t) = \sum_\xi t^{n(\xi)-n(\lambda)-n(\mu)}F_{\lambda\mu}^\xi(t^{-1})P_\xi(t). \]
Since the coefficients must lie in $\mathbb Z[t]$, we deduce that $F_{\lambda\mu}^\xi$ has degree at most $n(\xi)-n(\lambda)-n(\mu)$. Moreover, since the Hall-Littlewood polynomials specialise at $t=0$ to the Schur functions, we see immediately that the coefficient of $t^{n(\xi)-n(\lambda)-n(\mu)}$ in $F_{\lambda\mu}^\xi(t)$ is precisely the Littlewood-Richardson coefficient $c_{\lambda\mu}^\xi$.

A completely different proof of polynomiality is offered in \cite{Hubery3}, using the Hopf algebra structure\footnote{
More precisely, it uses Green's Formula \cite{Green1}, which is the formula needed to prove that the Ringel-Hall algebra is a twisted bialgebra.
}
in an intrinsic way. This approach allows one to quickly reduce to the case when $M_\xi$ is indecomposable, in which case the Hall number is either 1 or 0. One can then complete the result in the same manner as Zelevinsky above.

\subsection{Hall's Theorem}

We can use these polynomials to define the generic Ringel-Hall algebra $H_1$ over the ring $\bQ(v)$ of rational functions. This has basis $u_\lambda$ and
\begin{align*}
&\textrm{multiplication} & u_\lambda u_\mu &= \sum_\xi F_{\lambda\mu}^{\xi}(v^2)u_\xi,\\
&\textrm{comultiplication} & \Delta(u_\xi) &= \sum_{\lambda,\mu} F_{\lambda\mu}^{\xi}(v^2)\frac{a_\lambda(v^2)a_\mu(v^2)}{a_\xi(v^2)}u_\lambda\otimes u_\mu,\\
&\textrm{and Hopf pairing} & \langle u_\lambda,u_\mu\rangle &= \frac{\delta_{\lambda\mu}}{v^{4n(\lambda)}b_\lambda(v^{-2})}.
\end{align*}

\begin{Thm}[Steinitz, Hall, Macdonald]\label{Hall}
There is a monomorphism of self-dual graded Hopf algebras
\[ \Phi_1\colon \Lambda[t] \to H_1, \quad t\mapsto v^{-2}, \quad t^{n(\lambda)}P_\lambda(t)\mapsto u_\lambda. \]
\end{Thm}
Note that this induces an isomorphism with $\Lambda[t]\otimes\mathbb Q(t^{1/2})$.

The images of some of our special symmetric functions are given by
\begin{gather*}
e_r \mapsto v^{r(r-1)}u_{(1^r)}, \quad c_r(t) \mapsto (1-v^{-2})u_{(r)}, \quad h_r \mapsto \sum_{|\lambda|=r}u_\lambda,\\
p_r \mapsto \sum_{|\lambda|=r}(1-v^2)\cdots(1-v^{2\ell(\lambda)-2})u_\lambda.
\end{gather*}
We observe that $M_{(1^r)}$ is a semisimple, or elementary, module, and that $M_{(r)}$ is an indecomposable, or cyclic, module, so the terminology in these cases corresponds well.

\subsection{Proving Hall's Theorem}

We now describe one approach to proving Theorem \ref{Hall}. This is based upon ensuring that our map $\Phi_1$ is a monomorphism of self-dual Hopf algebras, rather than just an algebra map, and first finds candidates for the images of the cyclic functions $c_r(t)$, rather than the elementary symmetric functions $e_r$. This latter is the more common approach (see for example Macdonald \cite{Macdonald2} or Schiffmann \cite{Schiffmann6}), but requires the more difficult formula from Example (4). By starting with the cyclic functions, the formulae needed are much easier. We shall simplify notation by writing $q:=v^2$.

Consider the elements $u_{(r)}$, corresponding to the indecomposable modules $M_{(r)}$. Observe that each submodule and factor module of an indecomposable module is again indecomposable, and that $a_{(r)}=T^r(1-T^{-1})$. Therefore we can use Example (1) to deduce that
\[ \Delta(u_{(r)}) = u_{(r)}\otimes1+1\otimes u_{(r)}+\sum_{a=1}^{r-1}(1-q^{-1})u_{(a)}\otimes u_{(r-a)}. \]
We set
\[ X(T)=1+\sum_{r\geq1}x_rT^r := 1+(1-q^{-1})\sum_{r\geq1}u_{(r)}T^r, \]
so that
\[ \Delta(X(T))=X(T)\otimes X(T). \]

We next observe that the $x_r$ are algebraically independent and generate the Hall algebra. Setting $x_\lambda:=\prod_ix_{\lambda_i}$, we need to check that the $x_\lambda$ form a basis for $H_1$.

\begin{Prop}\label{x_lambda}
\[ x_\lambda = q^{n(\lambda)}b_\lambda(q^{-1})u_\lambda + \sum_{\mu>\lambda}\gamma_{\lambda\mu}u_\mu \quad\textrm{for some }\gamma_{\lambda\mu}. \]
\end{Prop}

\begin{proof}
Consider first the case $\lambda=(r^m)$. We prove by induction on $m$ that
\[ x_{(r^m)} = (1-q^{-1})^m u_{(r)}^m = q^{r\binom{m}{2}}\phi_m(q^{-1})u_{(r^m)}+\sum_{\mu>(r^m)}\gamma_{(r^m)\mu}u_\mu. \]
Multipying by $x_r$ we get
\[ x_{(r^{m+1})} = (1-q^{-1})q^{r\binom{m}{2}}\phi_m(q^{-1})u_{(r)}u_{(r^m)} + (1-q^{-1})\sum_{\mu>(r^m)}\gamma_{(r^m)\mu}u_{(r)}u_\mu. \]
Every summand $u_\xi$ satisfies $\xi\geq(r)\cup(r^m)$ by Example (5), so we just need to consider the coefficient of $u_{(r^{m+1})}$. This necessarily comes from the split exact sequence, so we can use Riedtmann's Formula together with
\[ \Ext^1(M_{(r)},M_{(r^m)})_{M_{(r^{m+1})}}=0 \]
to get
\begin{align*}
F_{(r)(r^m)}^{(r^{m+1})} &= \frac{a_{(r^{m+1})}}{q^{d((r),(r^m))}a_{(r)}a_{(r^m)}} = \frac{q^{r(m+1)^2}\phi_{m+1}(q^{-1})}{q^{rm}\cdot q^r(1-q^{-1})\cdot q^{rm}\phi_m(q^{-1})}\\
&= q^{rm}\frac{1-q^{-m-1}}{1-q^{-1}}.
\end{align*}
Hence the coefficient of $u_{(r^{m+1})}$ is $q^{r\binom{m+1}{2}}\phi_{m+1}(q^{-1})$ as claimed.

In general we can write $\lambda=(r^m)\cup\bar\lambda$ with $\bar\lambda_i<r$ for all $r$, so $x_\lambda=x_{(r^m)}x_{\bar\lambda}$. By induction we have the result for $x_{\bar\lambda}$, and by Example (5) we know that any summand $u_\mu$ must satisfy $\mu\geq(r^m)\cup\bar\lambda=\lambda$. So, we just need to consider the coefficient of $u_\lambda$, which again must come from the split exact sequence. Applying Riedtmann's Formula as before we deduce that
\[ F_{(r^m)\bar\lambda}^\lambda = q^{d((r^m),\bar\lambda)} = q^{m|\bar\lambda|}, \]
since $\bar\lambda_i<r$ for all $i$.

Now, $n(\lambda)=n(\bar\lambda)+m|\bar\lambda|+r\binom{m}{2}$ and $b_\lambda(q^{-1})=b_{\bar\lambda}(q^{-1})\phi_m(q^{-1})$, so the result follows.
\end{proof}

We immediately see that the $x_\lambda$ form a basis of the Ringel-Hall algebra. For, the $u_\lambda$ are by definition a basis for the Ringel-Hall algebra, and the proposition proves that the transition matrix from the $u_\lambda$ to the $x_\lambda$ for the partitions $|\lambda|=r$ is (with respect to a suitable ordering) upper triangular with non-zero diagonal entries.

Finally,
\[ \langle u_{(r)},u_{(r)}\rangle = (1-q^{-1})^{-1}, \quad\textrm{so}\quad \langle x_r,x_r\rangle = 1-q^{-1}. \]
It follows that we can define a monomorphism of self-dual graded Hopf algebras
\[ \Phi_1\colon\Lambda[t]\to H_1, \quad t\mapsto q^{-1}, \quad c_r(t)\mapsto x_r=(1-q^{-1})u_{(r)}. \]

We now compute the images of the other symmetric functions. For the complete symmetric functions, we use the formula
\[ \sum_{a=1}^rt^{r-a}h_{r-a}c_a(t) = (1-t^r)h_r. \]

\begin{Prop}
We have
\[ \sum_{a=1}^rq^{a-r}(1-q^{-1})\sum_{|\lambda|=r-a}u_\lambda u_{(a)} = (1-q^{-r})\sum_{|\xi|=r}u_\xi. \]
Thus
\[ \Phi_1(h_r) = \sum_{|\xi|=r}u_\xi. \]
\end{Prop}

\begin{proof}
The coefficient of $u_\xi$ on the left hand side is given by
\[ q^{-r}\sum_{a=1}^r\sum_\lambda q^a(1-q^{-1}) F_{\lambda(a)}^\xi. \]
Now, for fixed $a$ and $\xi$, we have seen that
\[ \sum_\lambda F_{\lambda(a)}^\xi = |\{U\leq M_\xi:U\cong M_{(a)}\}| = \frac{|\mathrm{Inj}(M_{(a)},M_\xi)|}{|\Aut(M_{(a)})|}. \]
If we write
\[ d_a:=d((a),\xi)=\dim\Hom(M_{(a)},M_\xi), \]
then this becomes
\[ q^a(1-q^{-1})\sum_\lambda F_{\lambda(a)}^\xi = q^{d_a}-q^{d_{a-1}}. \]
Thus the coefficient of $u_\xi$ on the left hand side is
\[ q^{-r}\sum_{a=1}^r\big(q^{d_a}-q^{d_{a-1}}\big) = q^{-r}\big(q^{d_r}-q^{d_0}\big) = q^{-r}(q^r-1) = 1-q^{-r}. \]
Here we have used that $d_0=0$, and that $d_r=r$ since the module $M_\xi$ has Loewy length at most $r$.
\end{proof}

We next consider the elementary symmetric functions, using the formula
\[ \sum_{a=1}^r(-1)^ae_{r-a}c_a(t) = (t^r-1)e_r. \]

\begin{Prop}
We have
\[ (1-q^{-1})\sum_{a=1}^r(-1)^aq^{\binom{r-a}{2}}u_{(1^{r-a})}u_{(a)} = (q^{-r}-1)q^{\binom{r}{2}}u_{(1^r)}. \]
Thus
\[ \Phi_1(e_r) = q^{\binom{r}{2}}u_{(1^r)}. \]
\end{Prop}

\begin{proof}
Using Example (2) when $a\geq 2$ and Example (3) when $a=1$, we can expand the left hand side to get
\begin{align*}
(1-q^{-1})&\sum_{a=1}^r(-1)^aq^{\binom{r-a}{2}}u_{(1^{r-a})}u_{(a)}\\
&= -q^{\binom{r-1}{2}-1}(q^r-1)u_{(1^r)}-q^{\binom{r-1}{2}}(1-q^{-1})u_{(1^{r-2}2)}\\
&\qquad\qquad + (1-q^{-1})\sum_{a=2}^r(-1)^aq^{\binom{r-a}{2}}\big(u_{(1^{r-a-1}a+1)}+q^{r-a}u_{(1^{r-a}a)}\big)\\
&= -q^{\binom{r-1}{2}+r-1}(1-q^{-r})u_{(1^r)}\\
&\qquad\qquad + (1-q^{-1})\sum_{a=2}^r(-1)^a\big(q^{\binom{r-a}{2}+r-a}-q^{\binom{r-a+1}{2}}\big)u_{(1^{r-a}a)}.
\end{align*}
Since $\binom{r-a}{2}+r-a=\binom{r-a+1}{2}$, this equals $-q^{\binom{r}{2}}(1-q^{-r})u_{(1^r)}$ as required.
\end{proof}

For the power sum functions we use the formula
\[ \sum_{a=0}^{r-1}(1-t^{r-a})p_{r-a}c_a(t) = rc_r(t). \]
For convenience we set
\[ y_r:=\sum_{|\lambda|=r}(1-q)\cdots(1-q^{\ell(\lambda)-1})u_\lambda. \]

\begin{Prop}
We have
\[ \sum_{a=1}^{r-1}(1-q^{a-r})y_{r-a}x_a = rx_r + (q^{-r}-1)y_r. \]
Thus
\[ \Phi_1(p_r)=y_r=\sum_{|\lambda|=r}(1-q)\cdots(1-q^{\ell(\lambda)-1})u_\lambda. \]
\end{Prop}

\begin{proof}
Substituting in for $y$ and $x$ and multiplying by $q^r$, we have on the left hand side
\[ \sum_{a=1}^{r-1}\sum_{|\lambda|=r-a}\sum_{|\xi|=r}(q^{r-a}-1)\cdot(1-q)\cdots(1-q^{\ell(\lambda)-1})\cdot q^a(1-q^{-1})F_{\lambda(a)}^\xi u_\xi. \]
We next observe that if $F_{\lambda(a)}^\xi\neq0$, then $0\leq\ell(\xi)-\ell(\lambda)\leq1$. Hence for $\xi\neq(r)$ we can divide the coefficient of $u_\xi$ by $(1-q)\cdots(1-q^{\ell(\xi)-2})$ to leave
\[ \gamma_\xi := \sum_{a=1}^{r-1}(q^{r-a}-1)\cdot q^a(1-q^{-1})\Big((1-q^{\ell(\xi)-1})\sum_{\ell(\lambda)=\ell(\xi)}F_{\lambda(a)}^\xi + \sum_{\ell(\lambda)<\ell(\xi)}F_{\lambda(a)}^\xi\Big). \]
We recall that for fixed $a$ and $\xi$, and setting $d_a := \dim\Hom(M_{(a)},M_\xi)$, we have
\[ q^a(1-q^{-1})\sum_\lambda F_{\lambda(a)}^\xi = |\mathrm{Inj}(M_{(a)},M_\xi)| = q^{d_a}-q^{d_{a-1}}. \]
Furthermore, given a short exact sequence
\[ 0\to M_{(a)}\to M_\xi\to M_\lambda\to 0, \]
we have that $\ell(\lambda)=\ell(\xi)$ if and only if the image of $M_{(a)}$ is contained in the radical of $M_\xi$. Since
\[ |\mathrm{Inj}(M_{(a)},\mathrm{rad}(M_\xi))| = q^{-\ell(\xi)}|\mathrm{Inj}(M_{(a+1)},M_\xi)| = q^{-\ell(\xi)}(q^{d_{a+1}}-q^{d_a}), \]
we can substitute in to $\gamma_\xi$ to get
\[ \gamma_\xi = \sum_{a=1}^{r-1}(q^{r-a}-1)\Big((q^{d_a}-q^{d_{a-1}})-q^{-1}(q^{d_{a+1}}-q^{d_a})\Big) = (1-q^r)(1-q^{\ell(\xi)-1}), \]
using that $d_0=0$, $d_1=\ell(\xi)$ and $d_r=r$. This shows that the coefficients of $u_\xi$ on the left and right hand sides agree for all $\xi\neq(r)$.

Now consider $\xi=(r)$. Then $F_{\lambda(a)}^{(r)}=\delta_{\lambda(r-a)}$, so the coefficient of $u_{(r)}$ on the left hand side equals
\[ (1-q^{-1})\sum_{a=1}^{r-1}(1-q^{a-r}) = r(1-q^{-1}) + (q^{-r}-1) \]
finishing the proof.
\end{proof}

Finally, we wish to show that

\begin{Prop}\label{HL-fns}
\[ \Phi_1(P_\lambda(t)) = q^{n(\lambda)}u_\lambda. \]
\end{Prop}

\begin{proof}
We begin by noting that
\[ \Phi_1(e_{\lambda'}) = \prod_i\Phi_1(e_{\lambda'_i}) = q^{\sum_i\binom{\lambda'_i}{2}}\prod_iu_{(1^{\lambda'_i})}=q^{n(\lambda)}\Big(u_\lambda+\sum_{\mu<\lambda}f_\lambda^\mu u_\mu\Big), \]
as mentioned in the discussion in Section \ref{HallPolys}.
Inverting this gives
\[ q^{n(\lambda)}u_\lambda = \Phi_1(e_{\lambda'}) + \sum_{\mu<\lambda}\tilde\beta_{\lambda\mu}\Phi_1(e_{\mu'}), \]
and since
\[ \langle q^{n(\lambda)}u_\lambda,q^{n(\mu)}u_\mu\rangle = \delta_{\lambda\mu}b_\lambda(q^{-1})^{-1}, \]
the result follows from our characterisation of the Hall-Littlewood symmetric functions.
\end{proof}

As promised, we can now deduce the formulae
\begin{Cor}
\begin{gather*}
e_r = P_{(1^r)}(t), \quad c_r(t) = (1-t)P_{(r)}(t), \quad h_r = \sum_{|\lambda|=r}t^{n(\lambda)}P_\lambda(t),\\
p_r = \sum_{|\lambda|=r}(1-t^{-1})(1-t^{-2})\cdots(1-t^{1-\ell(\lambda)})t^{n(\lambda)}P_\lambda(t).
\end{gather*}
\end{Cor}

\subsection{Integral Bases}

As for the ring of symmetric functions, we can also consider an integral version of the Ringel-Hall algebra. Since all the Hall polynomials have integer coefficients, we can consider the subring ${}_{\mathbb Z[q,q^{-1}]}H_1$ with $\mathbb Z[q,q^{-1}]$-basis the $u_\lambda$. Using Hall's Theorem, we deduce that $\Phi_1$ restricts to an monomorphism
\[ \Phi_1\colon {}_{\mathbb Z}\Lambda[t] \to {}_{\mathbb Z[q,q^{-1}]}H_1, \quad t\mapsto q^{-1}, \quad P_\lambda(t)\mapsto q^{n(\lambda)}u_\lambda, \]
and this induces an isomorphism with the ring ${}_{\mathbb Z}\Lambda[t,t^{-1}]$.

It follows that ${}_{\mathbb Z[q,q^{-1}]}H_1$ is generated either by the images of the elementary symmetric functions or by the complete symmetric functions
\[ \Phi_1(e_r)=q^{\binom{r}{2}}u_{(1^r)}, \quad \Phi_1(h_r)=\sum_{|\lambda|=r|}u_\lambda. \]

A very important basis is the canonical basis. This was introduced by Lusztig in \cite{Lusztig2}, and to define it we first need to introduce the bar involution. In Lusztig's geometric construction of the Ringel-Hall algebra he showed that it is natural to consider a weighted basis
\[ \tilde u_M := v^{\dim\End(M)-\dim M}u_\lambda. \]
This basis is said to be of PBW-type since it `lifts' the Poincar\'e-Birkhoff-Witt (PBW) basis for the universal enveloping algebra of the associated semisimple Lie algebra in the Dynkin case.

The bar involution is then defined via
\[ \overline v=v^{-1} \quad\textrm{and}\quad \overline{\tilde u_S^{\phantom M}}=\tilde u_S \quad\textrm{for all semisimple modules }S. \]
This defines a ring isomorphism and moreover
\[ \overline{\tilde u_M^{\phantom M}} = \tilde u_M + \sum_{M\leq_d N}\alpha_{MN}\tilde u_N, \]
where $\leq_d$ is the degeneration order on modules (see for example \cite{Bongartz}). Let $A$ be the matrix describing the transition from the $\tilde u_M$ to the $\overline{\tilde u_M^{\phantom M}}$ for modules of a fixed dimension. Then $A$ is upper-triangular with ones on the diagonal, and since the bar involution has order two, we must have $\overline AA=\mathrm{Id}$.

It follows that there is a unique upper-triangular matrix $B$ with ones on the diagonal and entries $\beta_{MN}(q^{-1})$ above the diagonal satisfying $\beta_{MN}(t)\in t\mathbb Z[t]$ and $B=\overline BA$. If we set
\[ b_M := \tilde u_M + \sum_{M<_d N}\beta_{MN}(q^{-1})\tilde u_N, \]
then the $b_M$ are bar invariant, $\overline{b_M}=b_M$. This is called the canonical basis, and is uniquely characterised by the two properties
\begin{enumerate}
\item[(a)] $b_M = \tilde u_M+\sum_{M<_dN}\beta_{MN}(q^{-1})\tilde u_N$ with $\beta_{MN}(t)\in t\mathbb Z[t]$
\item[(b)] $\overline{b_M}=b_M$.
\end{enumerate}

For the cyclic quiver $C_1$ we have $\dim\End(M_\lambda)-\dim M_\lambda=2n(\lambda)$ and the basis of PBW-type is given by $\tilde u_\lambda=q^{n(\lambda)}u_\lambda=\Phi_1(P_\lambda(t))$, so by the images of the Hall-Littlewood functions. Moreover, the degeneration order coincides with the opposite of the dominance order of partitions, and as in the proof of Proposition \ref{HL-fns} we have
\[ \Phi_1(e_{\lambda'}) = \tilde u_\lambda + \sum_{\mu<\lambda}q^{n(\lambda)-n(\mu)}f_\lambda^\mu\tilde u_\mu \quad\textrm{and}\quad \tilde u_\lambda = \Phi_1(e_{\lambda'}) + \sum_{\mu<\lambda}\tilde\beta_{\lambda\mu}\Phi_1(e_{\mu'}). \]
Since the bar involution fixes the semisimples $\tilde u_{(1^r)}=q^{\binom{r}{2}}u_{(1^r)}=\Phi_1(e_r)$ we have
\[ \overline{\tilde u_\lambda^{\phantom M}} = \tilde u_\lambda + \sum_{\mu<\lambda}\gamma_{\lambda\mu}\tilde u_\mu \]
as required.

We can therefore construct the canonical basis $b_\lambda$. This was done in \cite{Lusztig1} (see also \cite{Schiffmann4}).

\begin{Thm}
The canonical basis is given by the images of the Schur functions
\[ b_\lambda=\Phi_1(s_\lambda). \]
\end{Thm}

\begin{proof}
Set $b_\lambda:=\Phi_1(s_\lambda)$. We need to check that the $b_\lambda$ satisfy the two properties given above.

Since
\[ s_\lambda = P_\lambda(t) + \sum_{\mu<\lambda}K_{\lambda\mu}(t)P_\mu(t), \quad K_{\lambda\mu}(t)\in t\mathbb Z[t], \]
we see that
\[ b_\lambda = \tilde u_\lambda + \sum_{\mu<\lambda}K_{\lambda\mu}(q^{-1})\tilde u_\mu, \quad K_{\lambda\mu}(t)\in t\mathbb Z[t]. \]

Also, since the $\Phi_1(e_\lambda)$ are bar invariant and
\[ s_\lambda = e_{\lambda'} + \sum_{\mu<\lambda}\beta_{\lambda\mu}e_{\mu'}, \quad \beta_{\lambda\mu}\in\mathbb Z, \]
we see that the $b_\lambda$ are also bar invariant.
\end{proof}

Using the bilinear form we may also define the dual canonical basis $b_\lambda^\ast$. It immediately follows that
\[ b_\lambda^\ast = \Phi_1(S_\lambda(t)). \]

For reference we compute the first few canonical basis elements, using our earlier description of the Schur functions.
\begin{align*}
b_{(1)} &= u_{(1)} = \tilde u_{(1)}\\
b_{(1^2)} &= qu_{(1^2)} = \tilde u_{(1^2)}\\
b_{(2)} &= u_{(2)}+u_{(1^2)} = \tilde u_{(2)} + q^{-1}\tilde u_{(1^2)}\\
b_{(1^3)} &= q^3u_{(1^3)} = \tilde u_{(1^3)}\\
b_{(12)} &= qu_{(12)}+q(q+1)u_{(1^3)} = \tilde u_{(12)}+q^{-1}(1+q^{-1})\tilde u_{(1^3)}\\
b_{(3)} &= u_{(3)}+u_{(12)}+u_{(1^3)} = \tilde u_{(3)}+q^{-1}\tilde u_{(12)}+q^{-3}\tilde u_{(13)}\\
b_{(1^4)} &= q^6u_{(1^4)} = \tilde u_{(1^4)}\\
b_{(1^22)} &= q^3u_{(1^22)}+q^3(q^2+q+1)u_{(1^4)} = \tilde u_{(1^22)}+q^{-1}(1+q^{-1}+q^{-2})\tilde u_{(1^4)}\\
b_{(2^2)} &= q^2u_{(2^2)}+u_{(1^22)}+(q^4+q^2)u_{(1^4)} = \tilde u_{(2^2)}+q^{-3}\tilde u_{(1^22)}+q^{-2}(1+q^{-2})\tilde u_{(1^4)}\\
b_{(13)} &= qu_{(13)}+qu_{(2^2)}+q(q+1)u_{(1^22)}+(q^3+q^2+q)u_{(1^4)}\\
&= \tilde u_{(13)}+q^{-1}\tilde u_{(2^2)}+q^{-1}(1+q^{-1})\tilde u_{(1^22)}+q^{-3}(1+q^{-1}+q^{-2})\tilde u_{(1^4)}\\
b_{(4)} &= u_{(4)}+u_{(13)}+u_{(2^2)}+u_{(1^22)}+u_{(1^4)}\\
&= \tilde u_{(4)}+q^{-1}\tilde u_{(13)}+q^{-2}\tilde u_{(2^2)}+q^{-3}\tilde u_{(1^22)}+q^{-6}\tilde u_{(1^4)}
\end{align*}

\section{Cyclic Quivers, II}

We now generalise our discussion to larger cyclic quivers. In this case, there is a natural monomorphism from the ring of symmetric functions to the centre of the Ringel-Hall algebra, Theorem \ref{centre}.

Let $C_n$ be the cyclic quiver with vertices $1,2,\ldots,n$ and arrows $a_i\colon i\to i-1$ (taken modulo $n$).
\begin{center}
\newbox\ASYbox
\newdimen\ASYdimen
\def\ASYbase#1#2{\setbox\ASYbox=\hbox{#1}\ASYdimen=\ht\ASYbox%
\setbox\ASYbox=\hbox{#2}\lower\ASYdimen\box\ASYbox}
\def\ASYalign(#1,#2)(#3,#4)#5#6{\leavevmode%
\setbox\ASYbox=\hbox{#6}%
\setbox\ASYbox\hbox{\ASYdimen=\ht\ASYbox%
\advance\ASYdimen by\dp\ASYbox\kern#3\wd\ASYbox\raise#4\ASYdimen\box\ASYbox}%
\put(#1,#2){%
\wd\ASYbox 0pt\dp\ASYbox 0pt\ht\ASYbox 0pt%
\box\ASYbox%
}}
\def\ASYraw#1{#1}
\setlength{\unitlength}{1pt}
\includegraphics{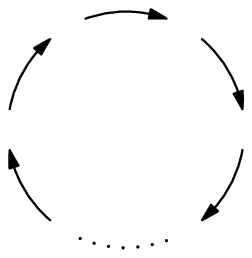}%
\definecolor{ASYcolor}{gray}{0.000000}\color{ASYcolor}
\fontsize{9.962640}{11.955168}\selectfont
\usefont{OT1}{cmr}{m}{n}%
\ASYalign(-112.412074,34.645115)(-0.500000,-0.320001){1.000000 0.000000 0.000000 1.000000}{$C_n\colon$}
\ASYalign(-78.268834,34.645115)(-0.500000,-0.500000){1.000000 0.000000 0.000000 1.000000}{$4$}
\ASYalign(-61.197214,64.213161)(-0.500000,-0.500000){1.000000 0.000000 0.000000 1.000000}{$3$}
\ASYalign(-27.053975,64.213161)(-0.500000,-0.500000){1.000000 0.000000 0.000000 1.000000}{$2$}
\ASYalign(-9.982355,34.645115)(-0.500000,-0.500000){1.000000 0.000000 0.000000 1.000000}{$1$}
\ASYalign(-27.053975,5.077069)(-0.500000,-0.500000){1.000000 0.000000 0.000000 1.000000}{$n$}
\ASYalign(-5.687135,12.452009)(-0.500000,-0.241629){1.000000 0.000000 0.000000 1.000000}{$a_1$}
\ASYalign(-5.687135,56.838221)(-0.500000,-0.241629){1.000000 0.000000 0.000000 1.000000}{$a_2$}
\ASYalign(-44.125594,79.031327)(-0.500000,-0.241629){1.000000 0.000000 0.000000 1.000000}{$a_3$}
\ASYalign(-82.564054,56.838221)(-0.500000,-0.241629){1.000000 0.000000 0.000000 1.000000}{$a_4$}
\end{center}
Define $\Cat_n(k)$ to be the category of $k$-representations of $C_n$, i.e. the category of functors from $C_n$ to finite dimensional $k$-vector spaces. Thus a representation $M$ is given by finite dimensional vector spaces $M(i)$ for each vertex $1\leq i\leq n$ and linear maps $M(a_i)\colon M(i)\to M(i-1)$ for each arrow $a_i$.

We can also view this as the category of finite dimensional modules over an hereditary order. Let $P=k[T^n]$ and set
\[ A_n:=\begin{pmatrix}
       P &       TP & T^2P & \cdots & T^{n-1}P\\
T^{n-1}P &        P &   TP & \cdots & T^{n-2}P\\
T^{n-2}P & T^{n-1}P &    P & \cdots & T^{n-3}P\\
\hdotsfor{5}\\
      TP &     T^2P & T^3P & \cdots & P
\end{pmatrix} \subset \mathbb M_n(P).\]
Thus, for $i\leq j$, we have the $P$-module $T^{j-i}k[T^n]$ in position $(i,j)$, whereas for $i>j$ we have $T^{n+j-i}k[T^n]$. The identification with $\Cat_n(k)$ is given as follows. Let $E_{ij}\in\mathbb M_n(P)$ be the standard basis. If $M$ is an $A_n$-module, then $M(i):=E_{ii}M$ and $M(a_i)$ is induced by the action of $TE_{i-1i}$. (See for example \cite{Geigle-Lenzing}.)

As before, define $\Cat_n^0(k)$ to be the full subcategory of nilpotent objects. These are functors $M$ such that the linear map
\[ M(a_1a_2\cdots a_n) := M(a_1)M(a_2)\cdots M(a_n)\in\End(M(n)) \]
acts nilpotently. Equivalently, these are those $R_n$-modules for which $T^n$ acts nilpotently.

The category $\Cat_n^0(k)$ has simple objects $M_i$ for $1\leq i\leq n$ where $M_i(j)=k^{\delta_{ij}}$ and $M_i(a_j)=0$. Moreover, this category is again a uniserial length category. We can index the isomorphism classes of indecomposable modules by pairs $(i;l)$ for $1\leq i\leq n$ and $l\geq1$, where the indecomposable $M_{(i;l)}$ has simple socle $M_i$ and length (or dimension) $l$. We therefore identify $i$ with $(i;1)$.

Given a partition $\lambda$ we set $M_{(i;\lambda)}:=\bigoplus_jM_{(i;\lambda_j)}$. More generally, given a multi-partition $\blambda=(\lambda^1,\ldots,\lambda^n)$ we set $M_{\blambda}:=\bigoplus_iM_{(i;\lambda^i)}$. This yields a bijection between the set of isomorphism classes and the set of multi-partitions. In particular, this set is combinatorial, so independent of the field $k$.

The category $\Cat_n^0(k)$ again satisfies our conditions, hence for each finite field $k$ we can define the Ringel-Hall algebra $H_n(k):=H(\Cat_n^0(k))$. We note that $K_0\cong\mathbb Z^n$ via $M_i\mapsto e_i$. Also, the radical of $(-,-)$ is generated by $\delta:=\sum_ie_i$. Thus $\overline K_0 = \mathbb Z^n/(\delta)\cong\mathbb Z^{n-1}$.

\subsection{Hall Polynomials}\label{HP2}

We have the following generalisation of Theorem \ref{Thm1}, due to Guo \cite{Guo}.
\begin{Thm}[Guo]
There exist
\begin{enumerate}
\item integers $d(\blambda,\bmu)$ such that, over any field $k$, $d(\blambda,\bmu):=\dim\Hom(M_{\blambda},M_{\bmu})$. In fact,
\[ \dim\Hom(M_{(i;l)},M_{(j;m)}) = |\{\max\{0,l-m\}\leq r<l:r\equiv j-i\bmod n\}|. \]
\item monic integer polynomials $a_{\blambda}$ such that, over any finite field $k$,
\[ a_{\blambda}(|k|) = a_{M_{\blambda}} = |\Aut(M_{\blambda})|. \]
\item integer polynomials $F_{\blambda\bmu}^{\bxi}$ such that, over any finite field $k$,
\[ F_{\blambda\bmu}^{\bxi}(|k|) = F_{M_{\blambda}M_{\bmu}}^{M_{\bxi}}. \]
Moreover,
\[ 2\deg F_{\blambda\bmu}^{\bxi}\leq h_{\bxi\bxi}-h_{\blambda\blambda}-h_{\bmu\bmu}. \]
\end{enumerate}
\end{Thm}

Guo's proof follows the approach outlined before by first showing that there exist polynomials $F_{\blambda\bmu}^{\bxi}$ whenever $M_{\blambda}$ is semisimple. Then, since the semisimple modules generate the Ringel-Hall algebra, we can complete the proof of existence of Hall polynomials in the same way.

Alternatively, the method given in \cite{Hubery3} works for all cyclic quivers, so we can again easily reduce to the case when $M_{\bxi}$ is indecomposable, where the result is trivial. Moreover, this method also allows one to obtain the upper bound for the degree of the Hall polynomials.

\subsection{Generalising Hall's Theorem}

As before, we can define the generic Ringel-Hall algebra $H_n$ over $\bQ(v)$. This has basis $u_{\blambda}K_\alpha$ with $\alpha\in\mathbb Z^n/(\delta)$, and we use the polynomials $F_{\blambda\bmu}^{\bxi}(v^2)$ instead of the integers $F_{M_{\blambda}M_{\bmu}}^{M_{\bxi}}$.

The analogue of Theorem \ref{Hall} says that there is a natural map from $\Lambda[t]$ to the centre $Z_n$ of $H_n$, and that $H_n$ is the tensor product of $Z_n$ with the composition algebra $C_n$. This is the subalgebra of $H_n$ generated by elements of the form $u_{\blambda}K_\alpha$ such that $\dim\Ext^1(M_{\blambda},M_{\blambda})=0$, or equivalently $\langle M_{\blambda},M_{\blambda}\rangle=h_{\blambda\blambda}$. If $n>1$, then $C_n$ is generated by $u_iK_\alpha$ for $1\leq i\leq n$. If $n=1$, then $C_1=\bQ(v)$.

\begin{Thm}[Schiffmann \cite{Schiffmann1}, Hubery \cite{Hubery1}]\label{centre}
We have
\begin{enumerate}
\item $H_n\cong Z_n\otimes C_n$ as self-dual graded Hopf algebras.
\item $Z_n=\bQ(v)[x_1,x_2,\ldots]$, where
\[ x_r := (-v^{-2})^{rn}\sum_{\substack{\blambda:[M_{\blambda}]=r\delta\\\mathrm{soc}(M_{\blambda})\textrm{ square-free}}} (-1)^{h_{\blambda\blambda}}a_{\blambda}(v^2)u_{\blambda}. \]
\item there is a natural monomorphism of self-dual graded Hopf algebras
\[ \Phi_n \colon \Lambda[t] \to Z_n, \quad t\mapsto v^{-2n}, \quad c_r\mapsto x_r. \]
\end{enumerate}
\end{Thm}

Again, $\Phi_n$ induces an isomorphism with $\Lambda[t]\otimes\bQ(t^{1/2n})$. Note also that $\mathrm{soc}(M_{\blambda})$ is square-free if and only if $\blambda=(m_1,\ldots,m_n)$ for some integers $m_i$. 

We remark that for $n=1$
\[ H_1=Z_1 \quad\textrm{and}\quad x_r = (1-v^{-2})u_{(r)} = \Phi_1(c_r), \]
so we recover Theorem \ref{Hall} as a special case.

Also, we should note that the definition of the $x_r$ given here differs by $\pm1$ to that originally given in \cite{Hubery1}. More precisely, the original definition had a factor of $(-1)^r$ at the front, which was chosen to ensure that the indecomposable $u_{(r,0,\ldots,0)}$ always had a positive coefficient. We now have the sign $(-1)^{rn}$, which means that the minimal elements $u_{(n,\ldots,n)}$ always have a positive coefficient. Our reasons for doing this will become clear in Conjecture \ref{Conj}.

We describe the first few such elements. Write $x_{n,r}$ for $x_r\in H_n$ and for convenience (and to save space) set $\alpha:=v^{-2}$, $\beta:=1-v^{-2}$ and abbreviate $u_{(m,n,p)}$ by $u_{(mnp)}$.
\begin{align*}
x_{1,1} &= \beta u_{(1)}\\
x_{1,2} &= \beta u_{(2)}\\
x_{2,1} &= -\alpha\beta\big(u_{(20)}+u_{(02)}\big)+\beta^2u_{(11)}\\
x_{2,2} &= \alpha^2\beta\big(u_{(40)}+u_{(04)}\big)-\alpha\beta^2\big(u_{(31)}+u_{(13)}\big)+\beta^2u_{(11)}\\
x_{3,1} &= \alpha^2\beta\big(u_{(300)}+u_{(030)}+u_{(003)}\big)-\alpha\beta^2\big(u_{(120)}+u_{(012)}+u_{(201)}\big)+\beta^3u_{(111)}\\
x_{3,2} &= \alpha^4\beta\big(u_{(600)}+u_{(060)}+u_{(006)}\big)-\alpha^3\beta^2\big(u_{(150)}+u_{(015)}+u_{(501)}\big)\\
&\quad-\alpha^3\beta^2\big(u_{(420)}+u_{(042)}+u_{(204)}\big)+\alpha^2\beta^3\big(u_{(411)}+u_{(141)}+u_{(114)}\big)\\
&\quad+\alpha^2\beta^2\big(u_{(330)}+u_{(033)}+u_{(303)}\big)-\alpha\beta^3\big(u_{(123)}+u_{(312)}+u_{(231)}\big)+\beta^3u_{(222)}
\end{align*}

It would be interesting to have expressions for the images of some of the other symmetric functions, for example $p_r$, $e_r$ and $h_r$.

\subsection{Outline of the Proof}

We will follow an approach by Sevenhant and Van den Bergh \cite{Sevenhant-VandenBergh2}, thus placing this result in a much broader context --- that of all extended Dynkin quivers.

For $n\geq2$ set $H:=H_n$ and recall that this is graded by $K_0=\mathbb Z^n$. We define a partial order on $K_0$ whereby $\alpha\geq\beta$ if and only if $\alpha_i\geq\beta_i$ for all $i$; here $\alpha=\sum_i\alpha_ie_i$ as usual. We can therefore write
\[ H = \bigoplus_{\alpha\geq0}H_\alpha \]
as a sum of its homogeneous parts. Note that
\[ H_0=\mathbb Q(v) \quad\textrm{and}\quad H_{e_i}=\mathbb Q(v)u_i. \]

We next define
\[ H'_\alpha := \Big(\sum_{\substack{\beta+\gamma=\alpha\\\beta,\gamma>0}}H_\beta H_\gamma\Big)^\perp \subset H_\alpha. \]
This definition is sensible since we can treat the symmetric bilinear form on $H$ like a positive definite form. For, the bilinear form on each $H(k)$ with $k$ a finite field is positive definite, as the $u_\lambda$ are pairwise orthogonal and each $\langle u_\lambda,u_\lambda\rangle$ is positive. Now, given any homogeneous $x\in H$, take a finite field $k$ such that $v_k=|k|^{1/2}$ is neither a zero or pole of the coefficients of $x$. Then we can specialise $x$ at $v_k$ to give a non-zero element $\tilde x$ of $H(k)$, whence $\langle\tilde x,\tilde x\rangle>0$.

We can therefore choose an orthogonal basis for each non-zero $H'_\alpha$, and if $\{\theta_j\}_{j\in J}$ is the union of these, then $H$ is generated as an algebra by the $\theta_j$.

Set $\alpha_j:=[\theta_j]\in K_0$. Then each $\theta_j$ is primitive:
\[ \Delta(\theta_j) = \theta_j\otimes1+K_{\alpha_j}\otimes\theta_j. \]
For, we can extend the $\{\theta_j\}$ to an orthogonal homogeneous basis $\{f_j\}$ for $H$. Then
\[ \Delta(\theta_j) = \theta_j\otimes1+K_{\alpha_j}\otimes\theta_j + \sum_{\substack{[f_a]+[f_b]=\alpha_j\\{}[f_a],[f_b]>0}}\gamma_{ab}f_aK_{[f_b]}\otimes f_b. \]
Now, since $[f_a],[f_b]<\alpha_j$, we have
\[ 0 = \langle\theta_j,f_af_b\rangle = \langle\Delta(\theta_j),f_a\otimes f_b\rangle = \gamma_{ab}. \]

We next show that $(\alpha_i,\alpha_j)\leq0$ for $i\neq j$. For, we have
\[ \Delta(\theta_i\theta_j) = \theta_i\theta_j\otimes1+\theta_iK_{\alpha_j}\otimes\theta_j+v^{(\alpha_i,\alpha_j)}\theta_jK_{\alpha_i}\otimes\theta_i+K_{\alpha_i+\alpha_j}\otimes\theta_i\theta_j. \]
From this it follows that
\[ \langle\theta_i\theta_j-\theta_j\theta_i,\theta_i\theta_j-\theta_j\theta_j\rangle = 2(1-v^{(\alpha_i,\alpha_j)})\langle\theta_i,\theta_i\rangle\langle\theta_j,\theta_j\rangle. \]
Specialising to an appropriate $H(k)$ for $k$ a finite field, the left hand side will be non-negative, and since $v_k>1$ we must have $(\alpha_i,\alpha_j)\leq0$ whenever $i\neq j$.

Now, we have already observed that $H_{e_i}=\mathbb Q(v)u_i$. Hence we may assume that each $u_i$ is in our set $\{\theta_j\}$. For any other $\theta_j$ we have $(\alpha_j,e_i)=0$ for all $i$, whence $\alpha_j\in\mathrm{rad}(-,-)=\mathbb Z\delta$. In this case we see from the calculation above that
\[ \langle\theta_i\theta_j-\theta_i\theta_j,\theta_i\theta_j-\theta_j\theta_i\rangle = 0 \quad \textrm{for all }\theta_i, \]
whence
\[ \theta_i\theta_j-\theta_j\theta_i=0 \]
and so each $\theta_j$ which is not of the form $u_i$ is central in the Ringel-Hall algebra.

We define
\[ C:=\mathbb Q(v)[\{u_i\}] \quad\textrm{and}\quad \overline Z:=\mathbb Q(v)[\{\theta_j\}\setminus\{u_i\}]. \]
Then $C$ is the composition subalgebra and $\overline Z$ is a central subalgebra. Both of these are graded Hopf subalgebras, since the $\theta_j$ are all homogeneous and primitive. They are also both self-dual, using that $\langle x,x\rangle=0$ implies $x=0$. Furthermore, $\overline Z$ is generated in degrees $r\delta$ for $r\in\mathbb N$.

Next we prove that there is a natural isomorphism
\[ H\cong \overline Z\otimes C. \]
For, $H$ is generated by the $\theta_j$ and $\overline Z$ is central, so the multiplication map $\overline Z\otimes C\to H$ is surjective. To see that this map is injective, we first note that if $x,x'\in\overline Z$ and $y,y'\in C$ are homogeneous, then
\[ \langle xy,x'y'\rangle = \langle x,x'\rangle\langle y,y'\rangle. \]
This follows since $\langle-,-\rangle$ is a Hopf pairing, and $\langle\overline Z,y\rangle=0$ for all homogeneous $y\in C$ with $0\neq[y]$ in $K_0$.

Now extend $\{\theta_j\}\setminus\{u_i\}$ to an orthogonal basis $f_j$ of $\overline Z$. If we have
\[ \sum_jf_jy_j=0 \quad\textrm{with }y_j\in C, \]
then
\[ 0 = \langle\sum_jf_jy_j,f_iy_i\rangle = \langle f_i,f_i\rangle\langle y_i,y_i\rangle, \]
whence $y_i=0$. Thus $H\cong\overline Z\otimes C$ as claimed. In particular we have the vector space decomposition
\[ H=\overline Z\oplus\big(\sum_iu_iH\big). \]

We can now deduce Schiffmann's characterisation
\[ \overline Z = \bigcap_i\mathrm{Ker}(e_i'), \]
where $e_i'$ are given by\footnote{
These operators are used by Kashiwara to define the crystal operators $\tilde e_i,\tilde f_i$ on the quantum group \cite{Kashiwara}.
}
\[ \langle e_i'(x),y\rangle := \langle x,u_iy\rangle. \]
Since the bilinear form is non-degenerate, this determines the operator $e_i'$ uniquely. Also, using that the form is a Hopf pairing, we have for homogeneous $x$ and $y$ that
\[ \Delta(x) = e_i'(x)K_i\otimes u_i + \textrm{other terms}, \]
where the other terms are linear combinations of the form $u_{\blambda}\otimes u_{\bmu}$ for $\bmu\neq i$, and
\[ e_i'(xy) = e_i'(x)y+v^{(u_i,x)}xe_i'(y). \]
It follows that $\bigcap_i\mathrm{Ker}(e_i')$ is a subalgebra of $H$.

Clearly each $\theta_j\neq u_i$ lies in $\mathrm{Ker}(e_i')$ since it is orthogonal to $u_iH$. Thus $\overline Z\subset\bigcap_i\mathrm{Ker}(e_i')$. On the other hand, if $x\in\bigcap_i\mathrm{Ker}(e_i')$, then we can write $x=\bar x+y$ with $y\in\sum_iu_iH$ and $\bar x\in\overline Z$. Hence $y=x-\bar x\in\bigcap_i\mathrm{Ker}(e_i')$, so $\langle y,y\rangle=0$, whence $y=0$.

So far we have shown that $H\cong\overline Z\otimes C$ with $\overline Z$ a central Hopf subalgebra generated in degrees $r\delta$ and $C$ the composition subalgebra, which is also a Hopf subalgebra. Moreover, $\overline Z=\bigcap_i\mathrm{Ker}(e_i')$.

We now use the explicit description of the $x_r$ to show that
\[ \langle e_i'(x_r),u_{\blambda}\rangle = \langle x_r,u_iu_{\blambda}\rangle = 0 \]
for all $u_{\blambda}$. Thus $x_r\in\bigcap_i\mathrm{Ker}(e_i')=\overline Z$, and in particular, they are all central.

Analogously to Proposition \ref{x_lambda} we compute that the minimal term in the product $x_\mu:=\prod_ix_{\mu_i}$ is precisely $u_{\bmu}$, where $\bmu=(\mu,\ldots,\mu)$. Using this we can prove that the $x_r$ are algebraically independent, that they generate the whole of the centre of $H$, and that
\[ \Delta(x_r) = \sum_{a+b=r}x_a\otimes x_b. \]
Finally, we calculate that
\[ \langle x_r,x_r\rangle = 1-v^{2n} \quad\textrm{for }r\geq1. \]

Thus
\[ \overline Z=\mathbb Q(v)[x_1,x_2,\ldots]=Z \]
is the centre of $H$, and there is a monomorphism of self-dual Hopf algebras
\[ \Phi_n\colon\Lambda[t]\to Z_n, \quad t\mapsto v^{-2n}, \quad c_r\mapsto x_r. \]

This completes the proof.

\subsection{Integral Bases}

Integral bases for the composition algebra $C_n$ have been given in \cite{Ringel6} in terms of condensed words, with a more thorough treatment given in \cite{Deng-Du} in terms of the more general distinguished words.

We would also like to study the canonical basis for the Ringel-Hall algebra $H_n$. We again set
\[ \tilde u_{\blambda} := v^{h_{\blambda\blambda}-|\blambda|}u_{\blambda}, \quad\textrm{where }|\blambda|=\sum_i|\lambda^i|, \]
which will be our basis of PBW-type.

Now, as mentioned in Section \ref{HP2}, the semisimple modules generate the Ringel-Hall algebra so we can again define the bar involution by specifying that it swaps $v$ and $v^{-1}$ and fixes each $\tilde u_{\blambda}$ whenever $M_{\blambda}$ is semisimple, which is if and only if each $\blambda=\big((1^{m_1}),\ldots,(1^{m_n})\big)$ for some integers $m_i$. In this case, however, it is non-trivial to deduce that this does indeed define a ring isomorphism. For, the semisimples $\tilde u_{\blambda}$ are no longer algebraically independent. This was, however, shown in \cite{Varagnolo-Vasserot}.

It is not hard to show that the transition matrix from any basis consisting of products of semisimples to the basis of PBW-type is upper-triangular with ones on the diagonal. It follows that
\[ \overline{\tilde u_{\blambda}^{\phantom M}} = \tilde u_{\blambda} + \sum_{\bmu<\blambda}\gamma_{\blambda\bmu}\tilde u_{\bmu}, \]
where $\bmu<\blambda$ if and only if $M_{\blambda}<_d M_{\bmu}$, analogous to the case for $n=1$.

We can therefore define the canonical basis $b_{\blambda}$ as before. A detailed study of this was done in \cite{Deng-Du-Xiao}. Also, it was shown in \cite{Schiffmann1} that the centre $Z_n$ also has a nice description in terms of the dual canonical basis. This is the basis $b_{\blambda}^\ast$ dual to the canonical basis with respect to the bilinear form $\langle-,-\rangle$. We call a multipartition $\blambda=(\lambda^1,\ldots,\lambda^n)$ aperiodic provided that for each $r$ some $m_r(\lambda^i)=0$; that is, some $\lambda^i$ contains no part of size $r$. At the other extreme, we call $\blambda$ completely periodic provided that $\lambda^1=\cdots=\lambda^n$.

\begin{Thm}[Schiffmann]
\[ Z_n = \bigoplus_\lambda \mathbb Q(v)b_{(\lambda,\ldots,\lambda)}^*, \]
where the sum is taken over all partitions, so the $(\lambda,\ldots,\lambda)$ are completely periodic multipartitions.
\end{Thm}

In his survey article \cite{Schiffmann4}, Schiffmann poses a question about the canonical basis and dual canonical basis elements corresponding to completely periodic functions. If we recall the vector space decomposition
\[ H_n = Z_n\oplus\big(\sum_iu_iH\big) \]
and let $\pi\colon H_n\to Z_n$ be the orthogonal projection onto the centre, then we know that for each partition $\lambda$, both $b_{(\lambda,\ldots,\lambda)}^\ast$ and $\pi(b_{(\lambda,\ldots,\lambda)})$ lie in $Z_n$.

We formulate the following conjecture, making Schiffmann's question more precise.

\begin{Conj}\label{Conj}
For each $n\geq1$ we have
\[ b_{(\lambda,\ldots,\lambda)}^\ast = \Phi_n(S_\lambda(t)) \quad\textrm{and}\quad \pi(b_{(\lambda,\ldots,\lambda)}) = \Phi_n(s_\lambda). \]
In particular,
\[ x_r = b_{(r,\ldots,r)}^\ast. \]
\end{Conj}

We know that $S_{(r)}(t)=c_r(t)$, so that $\Phi_n(S_{(r)}(t))=x_r$. This has minimal term corresponding to the completely periodic multipartition $(r,\ldots,r)$, and
\[ a_{(r,\ldots,r)} = q^{rn}(1-q^{-1})^n \quad\textrm{and}\quad h_{(r,\ldots,r)(r,\ldots,r)}=|(r,\ldots,r)|=rn, \]
so that
\[ x_r = (1-q^{-1})^nu_{(r,\ldots,r)} + \textrm{higher terms}, \]
where the higher terms correspond to multipartitions $\blambda>(r,\ldots,r)$. On the other hand, we know that
\[ b_{(r,\ldots,r)} = u_{(r,\ldots,r)} + \textrm{lower terms}, \]
where the lower terms correspond to multipartitions $\blambda<(r,\ldots,r)$. It follows that
\[ \langle x_r,b_{\blambda}\rangle = \delta_{\blambda(r,\ldots,r)} \quad\textrm{whenever }\blambda\leq(r,\ldots,r). \]
This offers some (minimal) support for the conjecture. One can however compute the first few canonical basis elements for $n=2$ and see that the conjecture does hold there. For example
\[ b_{(2,0)} = v^{-1}u_{(2,0)}+v^{-1}u_{(1,1)}, \quad b_{(0,2)} = v^{-1}u_{(2,0)}+v^{-1}u_{(1,1)}, \quad b_{(1,1)} = u_{(1,1)}. \]
Therefore, since
\[ x_1 = -q^{-1}(1-q^{-1})\big(u_{(20)}+u_{(02)}\big)+(1-q^{-1})^2u_{(11)}, \]
we see that
\[ b_{(1,1)}^\ast = x_1 = \Phi_2(S_{(1)}(t)). \]

We finish by mentioning that the crystal graph of the canonical basis for $H_n$ was determined in \cite{Leclerc-Thibon-Vasserot}. This describes how the operators $\tilde e_i$ (or equivalently $e_i'$) act on the canonical basis. The authors prove that the crystal graph decomposes into infinitely many components, labelled by the periodic multipartitions. Moreover, the aperiodic partitions form one connected component, isomorphic to the crystal graph of type $\tilde{\mathbb A}_{n-1}$, as was shown in \cite{Lusztig5}.

\end{document}